\documentclass[12pt]{amsart}
\usepackage{amssymb}
\usepackage{tikz-cd}
\usepackage[toc,page]{appendix}

\newtheorem{theorem}{Theorem}[section]
\newtheorem{lemma}[theorem]{Lemma}

\newtheorem{cor}[theorem]{Corollary}
\newtheorem{prop}[theorem]{Proposition}
\theoremstyle{definition}  
\newtheorem{example}[theorem]{Example}
\newtheorem{definition}[theorem]{Definition}

\newtheorem{remark}[theorem]{Remark}
\numberwithin{equation}{section}

\usepackage{pdflscape}

\usepackage{pgf}
\usepackage{tikz}
\usetikzlibrary{arrows,automata}
\usepackage[latin1]{inputenc}
\usepackage{verbatim}

\newcommand{\Z}{\mathbb Z}
\newcommand{\N}{\mathbb N}

\DeclareMathOperator{\Id}{Id}

\usepackage{color}

\newcommand{\ii}{{\mathfrak i}}
\usepackage{color}

\newcommand{\Aa}{{\mathcal A}}

\newcommand{\Cc}{{\mathcal C}}

\newcommand{\Ff}{{\mathcal F}}
\newcommand{\Gg}{{\mathcal G}}
\newcommand{\Hh}{{\mathcal H}}

\newcommand{\Mm}{{\mathcal M}}

\newcommand{\Rr}{{\mathcal R}}

\newcommand{\id}{{\mathbf 1}}

\newcommand{\R}{\mathbb{R}}
\newcommand{\Xe}{X_{eq}}

\newcommand{\aut}{{\mathfrak a}}
\newcommand{\aute}{\aut_{eq}}
\newcommand{\Xmax}{\Xe}
\newcommand{\pe}{\pi_{eq}}

\newcommand{\fb}[1]{\pe^{-1}(#1)}
\newcommand{\tpe}{\widetilde{\pe}}

\newcommand{\M}{\mathcal M}

\newcommand{\rst}{right-topological}

\newcommand{\im}{\mbox{\rm im}}

\newcommand{\Ef}{E^{fib}}

\newcommand{\Sfib}{\M^{fib}_0(X_\theta)}

\newcommand{\ev}{\mathrm{ev}}
\newcommand{\res}{\mathrm{res}}

\newcommand{\al}{\alpha}

\renewcommand{\a}{{\color{blue}\boxdot}}
\renewcommand{\c}{{\color{red}\boxminus}}
\renewcommand{\b}{{\color{orange}\boxplus}}

\newcommand{\stab}{\mathrm{stab}}
\newcommand{\tPhe}{\widetilde{\Phi_{eq}}}

\newcommand{\peq}{\pi_{eq}}

\newcommand{\Gf}{\Gg^{fib}} 

\newcommand{\cut}{\kappa}

\begin{document}
\title{When is the Ellis semigroup a complete conjugacy invariant?}
 \markboth{\hfill J. Kellendonk and R. Yassawi}{The Ellis semigroup as a complete invariant  \hfill }
\author{Johannes Kellendonk}
\address{Universit{\'e} Lyon 1, EC Lyon, INSA Lyon, Universit{\'e} Jean Monnet, CNRS, ICJ, UMR 5208, Lyon, France}
\email{kellendonk@math.univ-lyon1.fr}

\author{Reem Yassawi}
\address{ School of Mathematical Sciences , Queen Mary University of London, U.K.  
}
\email{r.yassawi@qmul.ac.uk}

\thanks{This work was supported by  the
   EPSRC  grant numbers EP/V007459/2. Reem Yassawi also thanks the CNRS for financial support for a longer visit.}


\begin{abstract} 
The Ellis semigroup of a topological dynamical system contains algebraic, topological and dynamical information. 
It is invariant under conjugacy. 
Despite this wealth of structure, two non-conjugate dynamical systems can have the same Ellis semigroup.
We identify a class of minimal dynamical systems inside which this cannot happen, that is, for which the Ellis semigroup is a complete conjugacy invariant.
\end{abstract}

\maketitle

\markboth{\hfill Johannes Kellendonk and Reem Yassawi}{The Ellis semigroup as a complete invariant \hfill }

\section{Introduction} {Let $T$ be an abelian group acting on a compact Hausdorff space $X$  by a continuous action $\sigma$. The
   Ellis semigroup $E(X)=E(X, \sigma,T)$ is   the compactification of the action with the topology of pointwise convergence.}
It is a right-topological monoid with an action $\sigma_E$ of $T$ given by left multiplication by $\sigma$. An {\em Ellis morphism}  is a factor map 
$\Phi: (E(X),\sigma_E,T) \rightarrow (E(X'), \sigma'_E,T)$ 
which is a monoid morphism.
Any factor map $\phi:(X, \sigma,T) \rightarrow (X', \sigma',T)$ between the dynamical systems induces an Ellis morphism. This suggests the question:
When is an Ellis morphism induced by 
a factor map? 
We identify in Theorem~\ref{thm-main} a class of topological dynamical systems for which this is always the case. 

Related to this question is the following: {When is an algebraic isomorphism between two Ellis semigroups an Ellis morphism?} Indeed, this work was motivated by the {fact} that the Thue-Morse shift and the 
Rudin-Shapiro shift have algebraically isomorphic Ellis semigroups, but that the shifts are not conjugate. A direct check that the algebraic isomorphism is continuous seems out of reach, since the Ellis semigroups are not tame and have very complicated topologies. As the class of dynamical systems for which the Ellis semigroup is a complete invariant contains the two shifts, our results show that the algebraic isomorphism between the Ellis semigroups of the Thue-Morse and the Rudin-Shapiro shift cannot be continuous. {To our knowledge, these questions have not been addressed in the literature.}

We focus in this work on dynamical systems $(X,\sigma,T)$ for which $X$ is compact Hausdorff, $T$ is abelian, and the action $\sigma$ is minimal and not distal. 
We include a section with examples of systems which satisfy at least some conditions of the class of systems to which our results can be applied.
In the second half of the paper we describe how to compute the algebraic structure of the Ellis semigroup for quasi-bijective primitive substitutions of constant length. This is a generalisation of earlier work on bijective substitutions \cite{ESBS}, and is required in order to describe the Ellis semigroup of the Rudin-Shapiro shift.

Let us give a rough overview of the strategy we use to extract from an 
Ellis morphism $\Phi:E(X)\to E(X')$ a factor map $\phi:X\to X'$ such that $\Phi=\phi_*$ (see Theorem~\ref{thm-main}).  
To exploit minimality, we restrict $\Phi$ to a minimal left ideal $L$ of $E(X)$, as $(L,\sigma_L)$ is a minimal component of $(E(X),\sigma_E)$, and $L'=\Phi(L)$  then defines a minimal component  of $(E(X'), \sigma'_E)$.   
Any factor map $\phi: (X, \sigma,T)\rightarrow (X', \sigma',T)$ between minimal systems is uniquely determined by its value at one point. In other words, given $x\in X$ and $x'\in X'$ there is at most one factor map $\phi$ such that $\phi(x) = x'$. Our aim is thus to obtain two such points $x$ and $x'$ from $\Phi$.

Recall that $\phi$ induces a factor map $\phi_{eq}: (\Xe, \delta,T)\rightarrow (X'_{eq}, \delta',T)$ between the maximal equicontinuous factors of the systems.  
If the system  has a  proximal pair (a pair of points which become arbitrarily close under the action of $T$), then there is a point $\xi\in \Xe$ whose fibre $\pe^{-1}(\xi)$ under the factor map $\pe:X\to \Xe$ contains this pair. We call such a $\xi$ a singular point. As factor maps preserve proximality, we see that if $\phi$ is a conjugacy then 
$\phi_{eq}$ identifies the singular points of $\Xe$ with those of $X'_{eq}$. 
To take advantage of this singular structure we make two essential assumptions. The first, which we call \emph{no extra spectrum} (NES), implies that the maximal equicontinuous factor of $(L,\sigma_L)$ is not larger than $\Xe$ (which in terms of continuous dynamical spectra means that the spectrum of $(L,\sigma_L)$ {equals} that of $(X,\sigma)$), and the second is that $\Phi_{eq}$ maps non-singular points to non-singular points. This latter condition is difficult to characterise in general, but we provide a sufficient criterion for it, namely that the singular points of $\Xe$ form a single orbit under the automorphism group of $(X,T)$.  In this case  we say that the system is a \emph{unique singular fibre} system (USF). 

Finally we need to transfer the information related to the image of points $x$ in singular fibres under the action of minimal idempotents of $L$.  Namely, we need each such point to separate the idempotents of $L$. Again this is difficult to guarantee in general, but if the automorphism groups are sufficiently rich, this is satisfied.

We summarise the contents of this paper. In Section~\ref{sec:preliminaries}, we set up notation, define some notions and describe the equicontinuous relation structure  of minimal ideals of $E(X)$. In Section~\ref{sec:main} we discuss when an Ellis morphism is induced by a factor map, and prove our main result, Theorem~\ref{thm-main}. In Section~\ref{sec:USF-systems}, we give classes of systems which will satisfy the conditions of Theorem~\ref{thm-main}. In Section~\ref{sec:Examples}, we discuss several examples, making use of the results in Section~\ref{ESQB}, where we give a complete algebraic description of the Ellis semigroup of quasi-bijective substitution shifts, needed for some of the examples in Section~\ref{sec:Examples}. Finally for the benefit of the readers (and authors!) we give a non-exhaustive glossary of notation in Section~\ref{sec:notation glossary}.

\section{Preliminaries}\label{sec:preliminaries}
\subsection{Dynamical systems}
Throughout this work, $T$ is an abelian group. 
We look at $T$-actions $ (X,\sigma,T)$ on a compact Hausdorff space $X$. If $T$ is understood we write $(X,\sigma)$ and if both $\sigma$ and $T$ are understood then we write $X$. If $X$ is metrisable and $d$ a metric on $X$ which generates its topology,  a pair of points $x,x'\in X$ are  {\em proximal} if  $\inf_{t\in T} d(\sigma^t(x),\sigma^t(x')) = 0$. This notion can also be formulated for Hausdorff compact spaces by using uniformity structures \cite{Auslander1988}. 
An action is \emph{distal} if the proximality relation is trivial in the sense that distinct points are never proximal. We make it a standing assumption that  our actions are not distal.

A dynamical system $(X, \sigma)$ is called  {\em equicontinous} if the family of homeomorphisms 
$\{\sigma^t : t\in T\}$ is equicontinuous.  If the action is transitive then this is the case if and only if, for any choice of $x_0\in X$ there is an abelian group structure on $X$ (denoted additively) such that $x_0$ is the identity element and $\sigma^t(x) = x+ \sigma^t(x_0)-x_0$. This group structure is topological.

An equicontinuous factor of $(X, \sigma)$ is a factor $\pi:(X, \sigma)\to (Y,\delta)$ such that $(Y,\delta)$ is equicontinuous. 
An equicontinuous factor of $(X,T)$ is {\em maximal} if any other equicontinuous factor of $(X,T)$ factors through it. Maximal equicontinuous factors (MEFs) always exist, and in this paper we will denote the maximal equicontinuous space by  $\Xe$  and the factor map by $\pe$. Given $\xi\in \Xe$ we call $\pe^{-1}(\xi)$ the fibre of $\xi$. 
It turns out that two proximal points are mapped to the same point under 
$\pe$, or stated differently, they belong to the same fibre, see, eg, \cite{ABKL}.

A point  $\xi \in \Xe$ is called {\em singular} if the fibre $\pe^{-1}(\xi)$ contains at least two proximal points, otherwise it is called regular. The assumption we make, that $(X,\sigma)$ is not distal, is to say that there is at  least one singular point in $\Xe$, and then of course, its orbit under $\delta$ consists of singular points. 
We choose for the group structure on $\Xe$ a neutral element which is a singular point and denote it by $0$. 

Any factor map $\phi: (X, \sigma,T)\rightarrow (X', \sigma',T)$ between minimal systems maps $\peq$-fibres of $(X,\sigma,T)$ to $\peq$-fibres of $(X', \sigma',T)$; this follows e.g.\ from \cite[Theorem 8, Chapter 9]{Auslander1988}. It hence 
induces a factor map $\phi_{eq}: (\Xe, \delta_X,T)\rightarrow (X'_{eq}, \delta_{X'},T)$, namely  $\phi_{eq}(\xi) = \pe(\phi(x)) - \pe(x)$ where $x$ is any point of $\pe^{-1}(\xi)$. 

 The {\em automorphism group} $\aut(X,\sigma)$ of a dynamical system $(X,\sigma)$ is the group, under composition, of all homeomorphisms of $X$ which commute with $\sigma$. Assuming minimality of 
$(X,\sigma)$ we denote $\aute:=\{\alpha_{eq}|\alpha\in \aut\}$ the corresponding automorphisms of $\Xe$. It can be seen as a subgroup of $\Xe$ or as the quotient group 
$\aut/\aut^{fib}$ where   $\aut^{fib}\subset \aut$ is the subgroup of $\aut$ of elements which preserve the fibres.

Let $\pe:(X,\sigma)\to (\Xe,\delta)$ be a maximal  equicontinuous factor. 
The {\em coincidence rank} $cr(\xi)$ of the fibre at $\xi\in \Xe$ is the largest possible cardinality that a subset of $\pe^{-1}(\xi)$ can have which only contains pairwise non-proximal elements. By minimality, $cr(\xi)$ does not depend on $\xi$ and we denote it by  $cr=cr(X)$.

\subsection{Matrix semigroups}\label{sec:matrix-semigroups} No semigroup discussed in this work has a zero element. 
A semigroup is called \emph{completely simple} if it does not admit a proper bilateral ideal and contains an idempotent. This implies that the semigroup is a disjoint union of isomorphic groups. The way that multiplication of  elements in different groups occurs is  encoded in a matrix. This is the result of the 
Rees-Suskevitch structure theorem which characterises completely simple semigroups as those which are isomorphic to {\em matrix semigroups} \cite{howie1995fundamentals}.   

Matrix semigroups are defined as follows.
Let $G$ be a group, let  $I$ and $\Lambda$ be  non-empty sets, and let   
$A = (a_{\lambda i})_{\lambda\in \Lambda,i\in I}$ be a  $\Lambda\times I$ matrix whose entries lie in $G$. Then the  {\em matrix semigroup} $M[G;I,\Lambda;A]$ 
 is the set $I\times G \times  \Lambda$ together with the multiplication
\begin{equation*} (i ,g,\lambda)(j,h,\mu) = (i, g a_{\lambda  j} h,\mu).\end{equation*}
The matrix $A$ is called the {\em sandwich matrix} and the group $G$ is called the {\em structure group}. The {\em little structure group}  of $M[G;I,\Lambda;A]$ is the subgroup $\Gamma$ of $G$ generated by the entries of the sandwich matrix $A$. \begin{example}[$\mathcal M_2$] \label{ex:ex}
The simplest matrix semigroup which is not orthodox, meaning that its idempotents do not form a subsemigroup, is obtained when $I$ and $\Lambda$ contain two elements, $G=S_2$ is the permutation group of two elements, and $A=\begin{pmatrix} e & e \\ e & \tau \end{pmatrix}$. Here $e$ is the identity and $\tau$ the transposition. This semigroup has $8$ elements of which $4$ are idempotents. We denote it by $\mathcal M_2$. Note that the little structure group of $\mathcal M_2$ coincides with the structure group $S_2$.
\end{example}

\subsection{Ellis semigroup} 
A compact  topological semigroup is a semigroup endowed with a compact topology, and 
such that  multiplication  $(x,y)\rightarrow xy$ is jointly continuous.
A compact right-topological semigroup is a semigroup endowed with a compact topology, and such that right multiplication  is continuous.  
Given a set $X$ we let $\Ff(X)$ denote the set of functions from $X$ to itself with the topology of pointwise convergence. Composition gives it the structure of a right-topological semigroup. If $X$ is compact then $\Ff(X)$ is a compact right-topological semigroup. It is  a compact topological semigroup if and only if $X$ is finite. 
Given a dynamical system $(X,\sigma, T)$, the family of homeomorphisms
$\{\sigma^t | t\in T\}$ is a subsemigroup of $\mathcal F(X)$. 
The closure of $\{\sigma^t | t\in T\}$, denoted $E(X, \sigma, T)$,  or simply $E(X)$ if the rest is understood, is still a semigroup, called the   {\em Ellis semigroup}  of the dynamical system. 
Since $X$ is compact 
$E(X, \sigma,T)$ is 
a compact \rst\ semigroup, by construction.  It is also a dynamical system, with the action $\sigma_E:E(X) \rightarrow E(X)$ defined by  $\sigma_E^t(f):= \sigma^{t}\circ f$. The system $(E(X), \sigma,T)$ is generally not minimal which is one of the reasons why we will work with a minimal left ideal $L$ of $E(X)$.

The Ellis semigroup is closely related to the proximality relation \cite[Chapter 3, Proposition 8]{Auslander1988}. 
Two points $x$ and $y$ are proximal if and only if there exists $f\in E(X)$ such that $f(x)=f(y)$. 
In particular we see that, given any idempotent $p\in E(X)$ and $x\in X$, the points $p(x)$ and $x$ are proximal.

Since the elements of $E(X)$ are limits of generalised sequences of powers of $\sigma$, the automorphism group, viewed as a subgroup of $\Ff(X)$, lies in the commutant of $E(X)$. 
 
\subsection{Morphisms of Ellis semigroups}

As an Ellis semigroup has an algebraic, a topological and a dynamical structure, the natural notion of a morphism between Ellis semigroups respects all the three.

\begin{definition}
An {\em Ellis  morphism} $\Phi: (E(X), \sigma_E)\rightarrow( E(X'), \sigma'_E)$ between the Ellis semigroups of two $T$-actions $(X,\sigma,T)$ and $(X',\sigma',T)$ is a topological  factor map which is a monoid morphism. More precisely, it satisfies \begin{enumerate}
\item $\Phi: E(X)\rightarrow E(X')$ is continuous,
\item $\Phi(fg) = \Phi(f) \Phi(g)$ and $\Phi (\id_X)= \id_{X'}$, 
\item
$\Phi(\sigma_E^{t}(f))= {\sigma'_E}^{t}\Phi(f)$ for each $t\in T$.
\end{enumerate}
 \end{definition}
If  such a morphism exists, then it is unique and we simply say that $(X',\sigma')$ is  related to $(X,\sigma)$ by an Ellis morphism. 
If $\Phi$  only satisfies (2) and (3) we call it an {\em algebraic morphism}. 

Any factor map  $\phi:(X, \sigma)\to (X',\sigma')$ induces an Ellis  morphism  
$\phi_*: E(X) \to E(Y)$ via $\phi_*(f)(y) = \phi(f(x))$ where $x$ is any pre-image of $y$ under $\phi$. The question we analyse here is the converse, namely when  an Ellis morphism is of the form $\Phi=\phi_*$ where $\phi$ is a factor map.

Given an Ellis morphism $\Phi:E\to E'$, 
to exploit minimality we restrict it to a minimal left ideal $L$. For that we
choose a minimal idempotent $e\in E(X)$ ($e$ is minimal in the sense that if $ep=pe=p$ for any other idempotent then $p=e$) and define $L=Ee$. 
While this and the definitions below depend on the choice of $e$, they do so only up to isomorphism. 

The minimal left ideal $L$ is a  compact subsemigroup and the action of $\sigma_E$ restricted to $L$ (which we denote by $\sigma_L$) is minimal \cite{Auslander1988}. 
Furthermore $\Gg := eL$ is the Rees structure group of the minimal bilateral ideal (the {\em kernel}) $\ker E$ of $E$ and we call it here simply the structure group\footnote{In the literature, the maximal equicontinuous factor of $(X,\sigma)$ is sometimes called the structure group. For almost automorphic systems our structure group reduces to that, but otherwise not.} 
 of $E$ or of the system. In general, $\Gg$ is not closed in $L$. A further group of interest is the little structure group $\Gamma$. As defined in Section~\ref{sec:matrix-semigroups},  it is the group generated by the matrix elements of the sandwich matrix of the matrix semigroup corresponding to  $\ker E$. It is also the intersection of the closure of the semigroup  $\langle J_{min}(E)\rangle $ generated by the minimal idempotents of $E$ with $\Gg$ \cite{ESBS}.

We define the same object in $E'$ using $e'=\Phi(e)$ as the idempotent, that is, we set $L'=E'e'$, $\Gg' = e'L'$ and the little structure group $\Gamma'$ as above. The following is straightforward, and can be found, eg, in \cite[Lemma 24]{BKY}.  
\begin{lemma}
Suppose  that $\Phi:E(X',\sigma')\rightarrow E(X,\sigma)$ is an Ellis morphism. Then 
$\Phi$ maps  $\Gg$ onto $\Gg'$ and $\Gamma$ onto $\Gamma'$.
\end{lemma}

\subsection{The fibre structure of $L$ and $\Gg$}\label{sec-2.4}

Let $\ev_x$ be evaluation of a function $f$ at the point $x$. We will use  the evaluation  map $\ev_x$ to evaluate  functions $f$  belonging to a variety of spaces; if needed, we will state the domain of $\ev_x$.
For a minimal equicontinuous system  $(Y,\delta)$ and w.r.t.\ the group structure on $Y$, whose neutral element we denote by $0$,
$\ev_{0} : E(Y) \to Y$ is an isomorphism of topological groups  \cite[Chap.~3, Theorem~6]{Auslander1988} (this hinges on our assumption that $T$ is abelian).
Let $x\in\pe^{-1}(0)$. By definition of the induced map ${\pe}_*:E(X)\to E(\Xe)$ the diagram
\begin{equation}\label{eq-commuting1}
\begin{matrix}
L & \stackrel{\ev_x}\to & X \\
{\pe}_* \downarrow & & \downarrow \pe \\
E(\Xe)  &\stackrel{\ev_0} \to & \Xe 
\end{matrix}
\end{equation}
commutes. Define
$$\tpe : = \ev_0\circ {\pe}_*:L\to \Xe.$$
Then we obtain the exact sequence of compact semigroups
$$L^{fib} \hookrightarrow L \stackrel{\tpe}\twoheadrightarrow \Xe$$
where $L^{fib}\subset L$ is the sub-semigroup of elements  
which preserve a $\peq$-fibre, and hence all $\peq$-fibres.
We can view an element $f\in L^{fib}$ as a function on $\Xe$ 
whose value at $\xi\in \Xe$ is given by the restriction of $f$ to  $e\pe^{-1}(\xi)$. As $L^{fib}$ commutes with the automorphism group we find that $\tilde f$ must satisfy the following so-called covariance condition
\begin{equation}\label{eq-cov-0}
\alpha \tilde f(\xi) \alpha^{-1} = \tilde f(\alpha_{eq}(\xi)),\quad \forall \alpha\in \aut.
\end{equation}
Let us define
\begin{equation}\label{eq-cov} 
Cov(\prod_{\xi\in \Xe} \Ff(e\pe^{-1}(\xi))):=\{(\tilde f(\xi))_\xi :\tilde f(\xi)\in \Ff(e\pe^{-1}(\xi)), \tilde f\: \mbox{ covariant}\}
\end{equation} 
which we equip with pointwise multiplication and the topology of pointwise convergence. We have an injective continuous semigroup morphism
$\mu:L^{fib}\to Cov(\prod_{\xi\in \Xe} \Ff(e\pe^{-1}(\xi))),$ defined by
\begin{equation}\label{def:mu}
\mu(f) = \tilde f,\quad \tilde f(\xi)(x) = f(x).\end{equation}Note that  $Cov(\prod_{\xi\in \Xe} \Ff(e\pe^{-1}(\xi)))$ is a  compact right-topological group.
 
Define $\Gg^{fib}:=e L^{fib}=\Gg\cap L^{fib}$, a subgroup of the semigroup $L^{fib}$. 
\begin{lemma} 
$\Gg^{fib}$ is a compact right-topological group. If the coincidence rank is finite, it is a topological group. 
\end{lemma}
\begin{proof} $\mu(\Gg^{fib})$ is the intersection of $\mu(L^{fib})$ with $Cov(\prod_{\xi\in \Xe} \Ff(e\pe^{-1}(\xi)))$, 
two compact right-topological groups. 
Moreover each $e\pe^{-1}(\xi)$ contains $cr$ elements. Hence $Cov(\prod_{\xi\in \Xe} \Ff(e\pe^{-1}(\xi)))$ is a compact topological semigroup if $cr$ is finite.
In that case $\mu(\Gg^{fib})$ and hence $\Gg^{fib}$ are also compact topological groups.
\end{proof}
As the restriction of $\tpe$ to $\Gg$ is still surjective onto $\Xe$, we obtain the exact sequence of groups
\begin{equation}\label{eq-SES}
\Gg^{fib} \hookrightarrow \Gg \twoheadrightarrow \Xe.
\end{equation} 
Note that, although  $\Gg^{fib}$, $\Xe$ are compact, $\Gg$ is typically not compact. 

Let $H$ be a subgroup of $G$. 
We denote the smallest normal subgroup of $G$ which contains $H$ by $\text{ncl}_G(H)$. If $G$ is a right-topological group we denote  its topological closure by 
$\overline{\text{ncl}_G(H)}$, and
we call $\overline{\text{ncl}_G(H)}$ the topological  normal closure of $H$ in $G$. Recall that $\Gamma$ is the little structure group of $\ker(E)$.

\begin{definition}\label{def:NES}
$(X,\sigma,T)$ {\em has no extra spectrum} (is NES) if the topological normal closure of $\Gamma$ in $\Gg$ is $\Gg^{fib}$, i.e.\ $\overline{\text{ncl}_{\Gg}(\Gamma)} = \Gg^{fib}$.
\end{definition}
 Corollary~\ref{cor:NES-equiv} makes this terminology clear.

\newcommand{\ol}{e_{eq}}

\subsection{The maximal equicontinuous factor of $(L,\sigma_L)$}
We investigate the maximal equicontinuous factor $\pe:L\to L_{eq}$ of $(L,\sigma_L)$. For better distinction we denote the neutral element of $L_{eq}$ by $\ol$ and fix the group structure on $L_{eq}$ such that $\pe(e)=\ol$. 
\begin{theorem}\label{thm:mef}  There is a closed normal subgroup $H$ of $\Gg$ which contains $\Gamma$ such that the semigroup morphism
$$L\ni f \mapsto ef\:mod\:H \in \Gg/H$$
is a maximal equicontinuous factor of $(L,\sigma_L)$.
\end{theorem} 
\begin{proof}
We first show that $\pe(g)=\ol$ for all $g\in \Gamma$. 
To show the claim, let $p\in J_{min}$ and let $(\sigma^{t_\nu})$ be a converging net such that $\lim {\sigma^{t_\nu}} = p$ in the topology of $E(X)$. 
Then $\lim {\sigma_L^{t_\nu}} = \lambda_p$, the left multiplication with $p$, in the topology of the Ellis semigroup $E(L)$ of $L$. As $p$ is an idempotent we have $\lambda_p f = \lambda_p pf$ for any $f\in L$. This shows that $f$ and $pf$ are proximal in $(L,\sigma_L)$ for all $f\in L$.
 As the equicontinuous structure relation of $(L,\sigma_L)$ contains the proximality relation \cite{Auslander1988} we find that $\pe(pf) = \pe(f)$ for  all $f\in L$. This implies that $\pe(g) = \pe(ge) = \pe(e) = \ol$ for all $g\in \langle J_{min} \rangle e$.
As $\pe$ is continuous this extends to the closure of $\langle J_{min} \rangle e$ in $L$ and hence $\pe(g) = \ol$ for all $g\in \Gamma$.

We claim that $\pe$ is multiplicative: Let $f,g\in L$. There is an idempotent $p\in L$ such that $pg=g$. As $L$ is closed and minimal there is a net $(t_\nu)_\nu$ such that $f=\lim \sigma^{t_\nu}p$. Thus
$$\pe(fg) =  \lim \pe(\sigma^{t_\nu} p g) =  \lim \sigma^{t_\nu} \pe(p g)$$
where we used continuity and $T$-equivariance of $\pe$.
Using $ \pe(p g) = \ol \pe(g) = \pe(p)\pe(g)$ we obtain
$$\pe(fg) =\lim \sigma^{t_\nu} \pe(p) \pe(g) = \pe(f)\pe(g). 
$$

The above shows that $\pe(f) = \pe(ef)$ for all $f\in L$ and that 
the restriction of $\pe$ to $\Gg$ is a continuous group epimorphism. 
 Thus there is a closed normal subgroup $H$ of $\Gg$, which contains $\Gamma$, so that $L_{eq} \cong\Gg/H$, and the isomorphism is given by $\pe(f) \mapsto ef$ mod $H$. 
\end{proof}

As $ \Xe$ is an equicontinuous factor of $L$, by universality of the maximal equicontinuous factor, there is a factor map $\eta_0$ making the following diagram commute, where to avoid ambiguity, we add an upper index to $\pe^L:L\rightarrow L_{eq}$.
\begin{equation}\label{eq-commuting2}
\begin{matrix}
L & \stackrel{\widetilde\pe}\to \,\, &\!\! \Xe \\
\!\!\!\pe^L  \downarrow & \:\:\nearrow \eta_0\!\! &  \\
L_{eq}  & & 
\end{matrix}
\end{equation}

\begin{cor}\label{cor:NES-equiv} If $(X,\sigma)$ has no extra spectrum, then $\eta_0$ is bijective, and $L^{fib} = {\pe^L}^{-1}(\ol)$.
\end{cor}
\begin{proof}
Suppose that $(X,\sigma)$ has no extra spectrum. Theorem~\ref{thm:mef} tells us that the (topologically closed) normal subgroup $H$ contains $\Gamma$.
 The assumption of no extra spectrum tells us  that the topological normal closure of $\Gamma$ is $\Gg^{fib}$, 
 i.e.\ $\overline{\text{ncl}_{\Gg}(\Gamma)} = \Gg^{fib}$.
 So $\Gg^{fib}$ is contained in $H$. But $H \subset \Gg^{fib}$ because $\Xe$ is a quotient of  $L_{eq}$.
  So $\eta_0$ is injective.

The second statement follows by (\ref{eq-commuting2}). 
\end{proof}
The corollary explains the naming ``no extra spectrum" as the injectivity of $\eta_0$ means that $L_{eq}$ and $\Xe$ are conjugate and hence $(X,\sigma)$ and $(L,\sigma_L)$ have the same 
topological dynamical spectrum. 

The usefulness of the property NES is that any Ellis morphism $\Phi:E\to E'$ between systems which have no extra spectrum gives rise to a group epimorphism  $\tPhe : \Xe \rightarrow X'_{eq}$, by
$$\tPhe := \eta_{0'}\circ \Phi_{eq} \circ \eta_0^{-1}$$
where $\Phi_{eq}:L_{eq}\to L'_{eq}$ is the morphism induced on the maximal equicontinuous factor. Furthermore,
\begin{lemma}\label{lem-fibs}
Suppose that each of $(X',\sigma')$, $(X,\sigma)$   have no extra spectrum. If 
 $(X,\sigma)$ and  $(X',\sigma')$, are related by an Ellis morphism $\Phi$,
 then $\Phi (\Gg^{fib})={\Gg'}^{fib}$. 
 \end{lemma}
\begin{proof}
As $X$ has no extra spectrum $\eta_0$ is injective so that $\Gg^{fib} = e\widetilde\pe^{-1}(0)=e{\pe^L}^{-1}(\ol)$. 
Clearly $\Phi(e{\pe^L}^{-1}(\ol)) \subset  e' {\pe^{L'}}^{-1}(\ol')$. 
\end{proof}

\section{When is an Ellis morphism induced by a factor map?}\label{sec:main}

Recall that we  fixed a minimal idempotent $e$ in $E(X)$ to define $L=E(X)e$, and moreover, if $(X',\sigma')$ is related to  $(X,\sigma)$  by an Ellis morphism $\Phi$ we set $e'=\Phi(e)$.
The evaluation map
$\ev_x:L\to X$,  given by $\ev_x(f) = f(x)$, is a factor map between $(L, \sigma_L)$ and $(X, \sigma)$. It defines a closed equivalence relation 
\begin{align}\label{eq:R-relation}R_{\ev_{x}}=\{(f_1,f_2)\in L\times L:f_1(x)=f_2(x)\}.\end{align} There is the analogous relation~\eqref{eq:R-relation} but defined for $L'=E'e'$.

\begin{lemma}\label{lem:morphisms-imply-factors}
 If
$(X',\sigma')$ is  related to $(X,\sigma)$ by an Ellis morphism $\Phi$, then
there is a factor map $\phi:(X,\sigma)\to (X',\sigma')$ such that $\phi_*=\Phi$  if and only if there are $x\in eX$, $x'\in e'X'$  
such that  $\Phi\times\Phi (R_{\ev_{x}})\subset R_{\ev_{x'}}$. 
\end{lemma}
\begin{proof}
Assume that there are $x\in eX$, $x'\in e'X'$  
such that $\Phi\times\Phi (R_{\ev_{x}})\subset R_{\ev_{x'}}$, i.e., if $f(x)=f'(x)$, then $\Phi (f(x)) = \Phi (f'(x))$. Given $y\in X$, all elements of  the preimage $\ev_x^{-1}(y)$ belong to the same $R_{\ev_x}$-class. Therefore 
$\phi(y):= \ev_{x'}(\Phi(\ev_x^{-1}(y))$ is well defined, and  
 the following diagram commutes:
$$
\begin{matrix}
L & \stackrel{\Phi}\to &L' \\
\ev_x \downarrow &  & \downarrow \ev_{x'} \\
X & \stackrel{\phi}\to & X'
\end{matrix}
$$
Applied to $e\in L$  and $x\in eX$, the diagram tells us that $\phi(x)=x'$. Since $\ev_x$, $\Phi$ and $\ev_{x'}$ are factor maps, $\phi$ is also a factor map, where its continuity follows from the continuity of $\Phi$. By minimality of the actions $\phi$ is the unique factor map satisfying the above.

We now show that $\Phi=\phi_*$, that is, $\Phi(f)(y') = \phi(f(y))$ for all $f\in E$, $y'\in X'$ and $y\in \phi^{-1}(y')$. The commuting diagram tells us that this equation is satisfied for all $f\in L$ if $y'=x'$ and $y=x$. Let $f\in E$, $y'\in X'$ and $y\in \phi^{-1}(y')$. By minimality of $\sigma$ there is $f_2\in L$ such that $f_2(x)=y$. Then
$\Phi(f)(y') = \Phi(f)(\phi(y)) = \Phi(f)(\phi(f_2(x))) = \Phi(f)\Phi(f_2)(x') = \Phi(f f_2)(x') = \phi(f f_2(x)) = \phi(f(y))$.

As for the converse, suppose that $\phi:X\to X'$ is a factor map and choose a minimal idempotent $e\in E$. Let $e'=\phi_*(e)$. Let $x\in eX$ and set  $x'=\phi(x)$. Then $e'(x') = \phi_*(e)(x') = \phi(ex) = x'$ so that $x'\in e'X'$. 
Let $(f_1,f_2)\in R_{\ev_x}$, i.e.\
$f_1(x)=f_2(x)$. As $\phi_*(f_1)(x')= \phi(f_1(x))$ we find 
$\phi_*(f_1)(x')=\phi_*(f_2)(x')$.
\end{proof}

\subsection{Conspiracy among the idempotents of $L$}

\begin{definition}
Let $J_L$ denote the idempotents of $L$.  We say that a point $x\in X$ {\em separates the idempotents} of $L$ if $\ev_x:J_L\to X$ is injective.
\end{definition}
Note that if the system is not equicontinuous, then a 
point in a regular fibre cannot separate the idempotents of $L$, as  any idempotent acts as the identity on that fibre. Therefore, 
the only other option would be that $L$ has a single idempotent. However in this case all fibres are regular and therefore the system is equicontinuous. 
 
 Recall that a semigroup is  completely simple  if it has no nontrivial bilateral ideal and contains an idempotent. A minimal left ideal of the Ellis semigroup is completely simple. In completely simple semigroups, every element $f$ has a unique normal inverse, i.e., an element $f^{-1}$ such that $ff^{-1}f = f$, $f^{-1}ff^{-1} = f^{-1}$ and $ff^{-1}=f^{-1}f$. 
Let $\Rr$ denote Green's right ideal relation, which for a completely simple semigroup means $(f_1,f_2)\in \Rr$ if $f_1^0 = f_2^0$, where 
$f^0 = f^{-1} f$. Note that the idempotent  $f^0$   satisfies $f^0f=f$. We can refine the relation $R_{\ev_x}$ defined in~\eqref{eq:R-relation}: for $\xi\in\Xe$, set
$$R^\xi_{\ev_x}(L)=\{(f_1,f_2)\in L\times L: (f_1,f_2)\in R_{\ev_x} \mbox{ and } {\pe}_* (f_i)= \xi, i=1,2\}.$$ 
Keep in mind that $\tpe (f_i)= \pe(f_i(x)) -\pe(x)$  for any $x$. 
When no confusion  arises we simply write $R^\xi_{\ev_x}$ for $R^\xi_{\ev_x}(L)$.
\begin{lemma}\label{lem-Rr} For $\xi\in\Xe$, the following are equivalent:
\begin{enumerate}
\item Each point in $\fb{\xi}$ separates the idempotents of $L$.
\item For each $x\in X$, $R^{\xi-\pe(x)}_{\ev_x}\subset \Rr$. 
\item For each $x\in \fb{\xi}$, $R^0_{\ev_x}\subset \Rr$. 
\end{enumerate}
\end{lemma}
\begin{proof} To see that (1) implies (2),  
let $x\in X$ and $(f_1,f_2)\in R^{\xi-\pe(x)}_{\ev_x}$. This means that \[f_2^0 f_2(x)= f_2(x)=f_1(x)=f_1^0 f_1(x) = f_1^0 f_2(x) , \] and $\pe(f_i(x))=\xi$. Assume that all points of $\fb{\xi}$ separate the idempotents of $L$. Then, as $f_2^0 f_2(x)=  f_1^0 f_2(x)$, we have $f_1^0 = f_2^0$. Hence $(f_1,f_2)\in\Rr$.

That (2) implies (3) is  obvious. 
Finally we show that (3) implies (1). Let  $x\in\fb{\xi}$ and 
let $p,q\in J_L$ with $p(x)=q(x)$. Then $(p,q)\in R^0_{\ev_x}$. If $R^0_{\ev_x}\subset \Rr$, then $p=p^0=q^0=q$. Hence $x$ separates $J_L$.
\end{proof}
\begin{lemma} \label{lem-sing}
Let $x\in eX$, $\xi\in \Xe$, and $L=Ee$.
\begin{enumerate}
\item If  $\xi$ is regular then
$(f_1,f_2)\in R^{\xi-\pe(x)}_{\ev_{x}}$ if and only if $(ef_1)^{-1}f_2 \in \stab_{\Gg}(x)$ and $\tpe(f_i) = \xi-\pe(x)$     . 
\item  
If each point of $\fb{\xi}$ separates the idempotents of $L$ then
$(f_1,f_2)\in R^{\xi-\pe(x)}_{\ev_{x}}$ if and only if $(ef_1)^{-1}f_2 \in \stab_{\Gg}(x)$, $f_1^0=f_2^0$ and $\tpe(f_i) = \xi-\pe(x)$. 
\end{enumerate}
\end{lemma}
\begin{proof} 
Note that $(f_1,f_2)\in R_{\ev_{x}}$ implies $ef_1(x) = e f_2(x)$, so
\[x=ex  =(ef_1)^{-1}ef_1(x) = (ef_1)^{-1}e f_2(x) = (ef_1)^{-1} f_2(x), \]
 with the last equality because $(ef_1)^{-1}\in Ee$. Conversely,
 if $x = (ef_1)^{-1} f_2(x)$ then $ef_1(x) = (ef_1)^0 f_2(x) = ef_2(x)$.
 
(1) The above shows the direction ``$\Rightarrow$". For the converse, assume that $\xi$ is regular, that  $ef_1(x)  = ef_2(x)$ and  that $
   \tpe(f_1) = \tpe(f_2) = \xi-\pe(x)$. Then $\pe(f_i(x)) =  \tpe(f_i)      +\pe(x) = \xi$ and, since $e$ acts as the identity on the fibre of $\xi$, we find  $f_1(x) =ef_1(x) = e f_2(x)=f_2(x)$.

(2) Assume that each point in  $\fb{\xi}$ separates the idempotents of $L$. By Lemma~\ref{lem-Rr} 
$R^{\xi-\pe(x)}_{\ev_{x}}\subset \Rr$ and so $(f_1,f_2)\in R^{\xi-\pe(x)}_{\ev_{x}}$ implies that   $f_1^0=f_2^0$. By the preliminary comment, $(ef_1)^{-1}f_2(x) = x$.

Conversely, if $f_1^0=f_2^0$ then 
$ef_1(x) = e f_2(x)$ implies \[f_1(x) = f_1^0  e f_1(x) = f_1^0 e f_2(x)= f_2^0 e f_2(x)=f_2(x),\]
so that $(f_1,f_2)\in R_{\ev_{x}}$.
\end{proof}
The assumption that all points in  $\fb{\xi}$ separate the idempotents of $L$, which can only be satisfied for singular $\xi$, is a pretty strong form of conspiracy. It implies that 
$\im p \cap \im q \cap \fb{\xi} =\emptyset$, i.e., $p \fb{\xi}\cap q \fb{\xi}=\emptyset$
whenever $p\neq q$ and $\xi$ is singular. If the system has finite coincidence rank, then 
$ p \fb{\xi}$ has $cr$ points; this means that the size of a singular fibre must be divisible by $cr$. This is not always the case: see for example the  bijective substitution\footnote{see Section~\ref{bijective} for the definition of a substitution shift.}
\[ a \mapsto abcca,\quad b\mapsto babab,\quad c\mapsto ccabc,\]
 for which $\pe^{-1}(0)=\{a.b,b.c,c.c,c.a,b.a\}$ has 5 elements, the 5 bi-infinite fixed points. The techniques in~\cite{ESBS} can be used to show that it has 2 minimal left ideals, each of which has 2 idempotents. The two idempotents both contain the fixed point $a.b$ in their image. 

The following theorem singles out three properties which guarantee that $\Phi$ comes from a factor map.  
We denote by $\Xe^{sing}$ to be the set of singular fibres and by $\Xe^{reg}$ the set of regular fibres in $\Xe$.
\begin{theorem}\label{thm-main} Consider two systems
$(X,\sigma)$ and $(X',\sigma')$ which have no extra spectrum and which are related by an 
Ellis morphism $\Phi:E(X)\rightarrow E(X')$. Let $L=E(X)e$ be a minimal left ideal and let $e'=\Phi(e)$. Let  $x\in e\pe^{-1}(0)$ and $x'\in e'\pe^{-1}(0)$. If
\begin{enumerate}
\item\label{as-1} all points in singular fibres of $\Xe$ separate elements of  $J_L$,
\item\label{as-3} $\tPhe ( \Xe^{\,reg}  ) \subset \Xe'^{\, reg}$, and
\item\label{as-2} $\Phi(\stab_{\Gg}(x))\subset \stab_{\Gg'}(x')$.
\end{enumerate}
then $x \mapsto x'$ extends to a factor map $\phi:X\to X'$ such that \ $\Phi=\phi_*$.  
\end{theorem}
\begin{proof}
We aim to show that  
$$\Phi\times\Phi (R_{\ev_{x}})\subset R_{\ev_{x'}}$$
for then by Lemma~\ref{lem:morphisms-imply-factors} we are done.
Let $(f_1,f_2)\in R_{\ev_{x}}$. 
As in the proof of  Lemma~\ref{lem-sing} we see that this implies $(ef_1)^{-1}f_2 \in \stab_{\Gg}(x)$. 
Under assumption (\ref{as-2}) 
we get $e'\Phi(f_1)^{-1}\Phi(f_2)\in {\stab_{\Gg'}(x')}$ 
hence
\begin{equation}\label{eq-1}
e'\Phi(f_1)(x')=e'\Phi(f_2)(x').
\end{equation}
It remains to show that this equation implies 
$\Phi(f_1)(x')=\Phi(f_2)(x')$.

We consider first the case that $ \tpe(f_i)$ is regular which is the same as saying that  $\pe(f_i(x))$ is regular, as $ \tpe(f_i)= \pe(f_i(x))-\pe(x)$ and $\pe(x)=0$.
Note that $ \tpe(f_1)= \tpe(f_2)$. As 
$\tpe(\Phi(f_1))=\tPhe(\tpe(f_1))$, $\tpe(\Phi(f_1))$ is regular by assumption (\ref{as-3})
and the first part of Lemma~\ref{lem-sing} allows to conclude  
that $\Phi(f_1)(x')=\Phi(f_2)(x')$.

Now suppose that $\xi:= \pe(f_i(x))$ is singular. Then, 
by Lemma~\ref{lem-Rr}, $(f_1,f_2)\in \Rr$. Since the $\Rr$-relation is preserved under morphisms we must have $(\Phi(f_1),\Phi(f_2))\in\Rr$.
Hence, by the second part of Lemma~\ref{lem-sing}, we get $\Phi(f_1)(x')=\Phi(f_2)(x')$. 
\end{proof}


\section{USF systems}\label{sec:USF-systems} 
The assumptions of Theorem~\ref{thm-main} are a bit elusive.    
The purpose of this section is to present a class of systems which satisfies them. The key properties of these systems which make this possible are, that the singular points of the system can be controlled by the automorphism group and that the automorphism group coincides with the center of the structure group of the Ellis semigroup. The latter property is a stronger form of {\em semi-regularity}, a notion that was recently proposed by Auslander and Glasner \cite{AG}.

As above, we chose a minimal idempotent $e$ of $E$ and identify the structure group with $\Gg=eEe$.
 The following does either not depend at all on the choice of $e$, or only up to isomorphism. 
\begin{definition} The system
 $(X,T)$ is called a \emph{unique singular fibre} (USF) system if  $\Xe^{sing}=\aute(0)$, i.e., $\Xe^{sing}$ is a single orbit under $\aute$. We call it \emph{strongly} USF if 
$X_0:=e\pe^{-1}(\Xe^{sing})$ is a single orbit under $\aut$. 
\end{definition}
If $cr=1$ then any USF system is strongly USF. 
We denote by $\aut_0$ the restriction of $\aut$ to $X_0$. Note that the restriction map is an isomorphism as by minimality an automorphism acts fixed point freely on its orbits. We also mention that $\Gg^{fib}$ is trivial if $cr=1$. So the following discussion is interesting only for systems with $cr>1$.

\subsection{The structure of $\Gg^{fib}$ for USF systems}
Recall that $\Ef$ are the elements of the Ellis semigroup which preserve the MEF fibres and that $\Gf = \Ef\cap \Gg $. We denote by 
 $\res_{e\pe^{-1}(\xi) }(f)$ the restriction of $f$ to the set $e\pe^{-1}(\xi)$,  for $f\in \Ef$, and we note that it is continuous.
We denote by  $\Gf_\xi$ the restriction of $\Gf$ to $e\pe^{-1}(\xi)$.
Likewise, we denote by $\Gamma_\xi$ the restriction of the little structure group $\Gamma$ to $e\pe^{-1}(\xi)$. The next lemma is \cite[Lemma 3.3]{ESBS}, using our notation here.
 \begin{lemma}
Each $\Gg^{fib}_\xi$ is isomorphic to $\Gf_0$.
\end{lemma}
\begin{proof}
Let $g\in \Gg$ such that $ \tpe(g) =\xi$. Then $g$ maps $e\pe^{-1}(0)$ bijectively onto $e\pe^{-1}(\xi)$. It follows that conjugation of $\Gg^{fib}_0$ by this bijection is $\Gg^{fib}_\xi$.
\end{proof}

Recall the definition of $Cov(\prod_{\xi\in \Xe} \Ff(e\pe^{-1}(\xi)))$ from Section~\ref{sec-2.4}. 
It has the following subgroup
$$Cov(\prod_{\xi\in \Xe} \Gg^{fib}_\xi):=\{(\tilde g(\xi))_\xi:\tilde g(\xi)\in \Gg^{fib}_\xi, g\:\:\mbox{is covariant}\}.$$
Restricting $\mu$ as defined in \eqref{def:mu}, which we denote by the same letter, we obtain a continuous injective group morphism
$\mu:\Gg^{fib} \to Cov(\prod_{\xi\in \Xe} \Gg^{fib}_\xi)$, 
$\mu(g) = \tilde g$ with $\tilde g(\xi)(x) = \tilde g(x)$ for all $x\in \pe^{-1}(\xi)$. 
\begin{lemma}\label{lem-little}
If $(X,\sigma,T)$ is a USF system, then the  little structure group $\Gamma$ acts trivially on $X\backslash X_0$. Moreover, the restriction $\Gamma_0$ of $\Gamma$ to $e\pe^{-1}(0)$ 
is isomorphic to $\Gamma$.
\end{lemma}
\begin{proof}
As idempotents act trivially on regular points, so  $J_{min}$ acts trivially on the complement of $X_0$. By covariance the action of an idempotent on $e\pe^{-1}(0)$ determines uniquely its action on $e\pe^{-1}(\xi)$ for all $\xi \in \aute(0)$.
\end{proof}
The following theorem generalises some statements in 
\cite[Section 3]{ESBS}. In particular, we relax the requirement that the only singular fibres are a  $\mathbb Z$-orbit in $\Xe$,  demanding instead that it is a USF system, and obtain results analogous to \cite[Corollary 3.7]{ESBS}. The requirement that a system has no extra spectrum allows us to obtain results similar to \cite[Thm 3.14]{ESBS}. 

\begin{theorem} \label{thm-prod}
If $(X,\sigma,T)$ is a USF system, then the following are equivalent:
\begin{enumerate}
\item $(X,\sigma,T)$ has no extra spectrum,
\item 
$  \overline{\text{ncl}_{\Gg^{fib}_0}(\Gamma_0)} =\Gg^{fib}_0$.
\end{enumerate}
Under these conditions
\begin{enumerate}
\item[(3)] $\mu$ is surjective and
$$\Gg^{fib} \cong    (\Gg^{fib}_0)^{\Xe/\aute} .  $$
\end{enumerate}
\end{theorem}

\begin{proof}
``$1 \Rightarrow 2$":
Suppose that $(X,\sigma)$ has no  extra spectrum so that  $\overline{\text{ncl}_{\Gg}(\Gamma)}=\Gg^{fib}$.  We transfer the information in $\overline{\text{ncl}_{\Gg}(\Gamma)}$ to information about  $\overline{\text{ncl}_{\Gg^{fib}_0}(\Gamma_0)} $ as follows.

    The elements of $\text{ncl}_{\Gg}(\Gamma)$ are products of elements of the form $g \gamma g^{-1}$ where $\gamma\in \Gamma$ and $g\in\Gg$. We look at one such element. If $\gamma$ is nontrivial, then, by Lemma~\ref{lem-little},
    $\res_{e\pe^{-1}(0) }(g \gamma g^{-1})$ acts non-trivially only if $\tpe(g)\in \aute(0)$.  If $g \in \aut(X)$ then it commutes with $\gamma$ so $  \res_{e\pe^{-1}(0) }(g \gamma g^{-1}) = \res_{e\pe^{-1}(0) }( \gamma )$. If $\tpe(g)\in \aute(0)$ but 
 $g\notin\aut$ we choose $f\in\aut(X)$ such that $\tpe(g)=\tpe(f)$. Then  $g \gamma g^{-1} = g  f ^{-1}\gamma f g^{-1}$ and $\tpe(fg^{-1})=0$. This gives the first equality in
    \begin{equation}\label{eq:big-to-small}\res_{e\pe^{-1}(0) }(\text{ncl}_{\Gg}(\Gamma))=\res_{e\pe^{-1}(0) }(      \text{ncl}_{\Gg^{fib}}(\Gamma)                     )=             \text{ncl}_{\Gg^{fib}_0}(\Gamma_0)       ,\end{equation} with the second equality coming from the fact that the restriction of $\Gamma$ to $\pe^{-1}(0)$  is $\Gamma_0$.
     As the map $\res_{e\pe^{-1}(0) }$ is continuous, we have $\res_{e\pe^{-1}(0) }(\overline{\text{ncl}_{\Gg}(\Gamma)})= \overline{\res_{e\pe^{-1}(0) }(\text{ncl}_{\Gg}(\Gamma))}$, so that (1) 
 and ~\eqref{eq:big-to-small} gives
    $$ \Gg^{fib}_0 = \res_{e\pe^{-1}(0) }(   \Gg^{fib}      )= \res_{e\pe^{-1}(0) }(\overline{\text{ncl}_{\Gg}(\Gamma)})=    \overline{ \text{ncl}_{\Gg^{fib}_0}(\Gamma_0)}.$$

``$2 \Rightarrow 3$":      
Consider an element  of the form $f:=g\gamma g^{-1}$, $\gamma\in\Gamma$, $g\in \Gg$. As before we see that $f$ only acts nontrivially on $e\pe^{-1}(\xi)$ if $\xi$ belongs to  $\aute$-orbit of ${\pe}_*(g)$. In other words, $f$ acts trivially on the complement of $X_{{\pe}_*(g)}$. Moreover, since it commutes with the automorphism group, the action of $f$ on $X_{{\pe}_*(g)}$ is uniquely determined by its action on $e\pe^{-1}({\pe}_*(g))$. We have seen that our assumption that the system has no extra spectrum is equivalent to  
$   \overline{   \text{ncl}_{ \Gg_0^{fib} }(\Gamma_0) }    = \Gg^{fib}_0$. 
Thus
$\res_{e\pe^{-1}(\xi) }(g   \,   \overline{\text{ncl}_{\Gg^{fib}_0}(\Gamma_0)}  \, g^{-1}) = \Gg^{fib}_\xi$. Therefore, given any $h\in \Gg^{fib}_\xi$ there exists $f\in\Gg^{fib}$ such that $\mu(f)(\xi) = h$ and   $\mu(f)(\xi') = \mathrm{id}$
provided $\xi'\notin\aute(\xi)$. Given any finite subset $\{\xi_1,\cdots,\xi_k\}$ of $\Xe$ which contains from each $\aute$-orbit at most one point and for each $\xi\in F$ a choice of element $h_i\in \Gg^{fib}_{\xi_i}$, taking the product $f=\prod_i f_i$ of the corresponding $f_i$ we obtain an element which, restricted to $X_{\xi_i}$ act as the chosen element $h_i$, that is, $\mu(f)(\xi_i) = h_i$. Taking the topological closure of such products we obtain all possible covariant functions, that is, $\mu$ is surjective. Continuity of $\mu$ is clear.
Since all groups $\Gg^{fib}_\xi$ are isomorphic to  $\Gg^{fib}_0$
we find  $\Gg^{fib} \cong  (\Gg^{fib}_0)^{\Xe/\aute}$.

``$2  \Rightarrow 1$":
Suppose that 
$  \overline{\text{ncl}_{\Gg^{fib}_0}(\Gamma_0)} =\Gg^{fib}_0$.
   We just showed that this implies $\Gg^{fib} \cong     \overline{   \text{ncl}_{ \Gg_0^{fib} }(\Gamma_0) }       ^{\Xe/\aute}$. With the USF condition, this shows that 
 $         \text{ncl}_{ \Gg^{fib} }(\Gamma) =  \text{ncl}_{ \Gg_0^{fib} }(\Gamma_0)        \times \id $                 where $\id$ is the identity on $\prod_{[\xi]\neq [0]\in \Xmax/\aute} \Gg^{fib}_{[\xi]}$. The direct product decomposition also shows that 
$\overline{ncl_\Gg(\Gg^{fib}_0\times\id)}=\Gg^{fib}$. Finally, 
      $$ncl_\Gg  (\Gamma)=ncl_\Gg (ncl_{   \Gg^{fib}  }(\Gamma)) = ncl_{\Gg } (\Gg^{fib}_0\times\id)$$
allows us to conclude that there is no extra spectrum.
\end{proof}
\subsection{The algebraic structure of $\Gg$ for USF systems}  
Recall that the structure group $\Gg$ fits into the exact sequence
\begin{equation}\label{eq-SESg}
\Gg^{fib} \hookrightarrow  \Gg \stackrel{\widetilde\pe} \twoheadrightarrow \Xe \, .
\end{equation}
We now use the information about $\Gg^{fib}$ to determine the algebraic structure of $\Gg$ in the case that $\Gg_0^{fib}$ is maximal, that is, contains all permutations of $e\pe^{-1}(0)$. Unlike in \cite{ESBS} we do not use a splitting section for \eqref{eq-SESg} here, as the existence of such a section depends on  the axiom of choice and we do not have an explicit construction of it. 
Let $\Gg^\xi = \{f\in \Gg: \widetilde\pe (f) = \xi\}$ so that $\Gg = \bigsqcup_{\xi\in \Xe} \Gg^\xi$. In particular, $\Gg^0=\Gg^{fib}$.
Let $B(\eta,\xi)$ denote the set of bijective maps from $e\pe^{-1}(\eta)$ to $e\pe^{-1}(\xi)$. We can identify an element $g\in \Gg^\xi$ with the function $\tilde g$ on $\Xe$ whose value at $\eta$ is the reduction of  $g$ to  
a bijection from $e\pe^{-1}(\eta)$ to $e\pe^{-1}(\eta+\xi)$,
$$\tilde g(\eta) = g\left|_{ e\pe^{-1}(\eta)\to e\pe^{-1}(\eta+\xi)}\right.  \in B(\eta,\eta+\xi).$$
Furthermore, $\tilde g$ is covariant in the sense that it satisfies \eqref{eq-cov-0}. 
\begin{lemma}\label{lem-alg-iso}
Let $(X,\sigma)$ be a USF system with no extra spectrum. If  $\Gg_0^{fib}$ is maximal, then
 $\Gg$ is algebraically isomorphic to the set of all covariant functions $\tilde g$ on $\Xe$ such that $\tilde g(\eta) \in B(\eta,\eta+\xi)$, $\xi= \widetilde\pe (g)$,
with product structure 
$$(\tilde g\tilde g')(\eta) = \tilde g(\eta+\zeta) \tilde g'(\eta).$$
\end{lemma}
\begin{remark}We emphasise that this does not say anything about the topology of $\Gg$. \end{remark}
\begin{proof} 
It is easily seen that,
when identifying the elements of $\Gg = \bigsqcup_{\xi\in \Xe} \Gg^\xi$ 
with covariant functions on $\Xe$ as described above, the product of such functions becomes the one stated above. We claim that the identification is surjective. 
 By Theorem~\ref{thm-prod},  $\Gg^{fib}$ consists of all covariant functions with values in $\Gg_0^{fib}$. 
As $\Gg$ is invariant under left (and right) multiplication by elements of $\Gg^{fib}$, and $\Gg_0^{fib}$ contains all  permutations of $e\pe^{-1}(0)$, we see that all possible bijections of $B(\eta,\eta+\xi)$ can arise. 
\end{proof}

\subsection{The algebraic structure of the kernel $\M(X)$}
Recall that the kernel $\M(X)$ of $E(X)$ is its smallest bilateral ideal. It is a completely simple semigroup. By the Rees-Suskevitch theorem, it is isomorphic to a matrix semigroup $\M(I,\Gg,\Lambda,\tilde A)$ whose structure group is the group $\Gg$ discussed above, and which we also call  the structure group of the system. Furthermore,  $\M(X)$ contains the sub semigroup $\M^{fib}(X):=\M(X)\cap \Ef(X)$, which itself is completely simple and thus is isomorphic to a matrix semigroup. The two matrix semigroups are related. We claim that $\M^{fib}(X)$ is isomorphic to 
$\M(I,\Gg^{fib},\Lambda,\tilde A)$.\footnote{We mention that this result needs neither minimality nor USF.}  This follows from the fact that the 
sets $I$ and $\Lambda$, which label the right and the left ideals respectively, and the sandwich matrix $\tilde A$, are uniquely determined by the idempotents of $\M(X)$ and their products, and these lie in $\M^{fib}(X)$. In particular, the sandwich matrix $\tilde A$ takes values in the subgroup $\Gg^{fib}$ of $\Gg$. Related to that is the result that the little structure group $\Gamma$ is generated (as a group) by the entries of $\tilde A$ \cite{ESBS}.

For USF systems 
the situation is even simpler. The restriction of $\M^{fib}(X)$ to the fibre $\pe^{-1}(0)$, which we denote by $\M_0^{fib}(X)$, is also a completely simple semigroup and therefore isomorphic to a matrix semigroup. The idempotents act non-trivially only on the fibres $\pe^{-1}(\xi)$ with $\xi$ in the $\aute$-orbit of $0$. The sets $I$ and $\Lambda$ and the sandwich matrix $\tilde A$ can be determined through their action on $\pe^{-1}(0)$. $\M_0^{fib}(X)$ is therefore isomorphic to 
$\M(I,\Gg^{fib}_0,\Lambda,A)$, where $I$ and $\Lambda$ are the same as above and $\tilde A=(\tilde a_{\lambda,i})_{\lambda,i}$ is determined by $A=(a_{\lambda,i})_{\lambda,i}$ in the following way:
The restriction of  $\tilde a_{\lambda,i}$ to the fibres $\pe^{-1}(\xi)$ with $\xi$ in the $\aute$-orbit of $0$ is equal to $a_{\lambda,i}$ whereas otherwise 
$\tilde a_{\lambda,i}=e$.

 The above discussion,  together with Theorem~\ref{thm-prod}, gives
\begin{prop}\label{prop-alg-kernel}
Let $(X,\sigma)$ be a  USF system for which the topological  normal closure of $\Gamma_0$ contains all elements of the permutation group $S_{\pe^{-1}(0)}$ of $\pe^{-1}(0)$. Then the algebraic structure of the kernel $\M(X)$ of the Ellis semigroup is entirely determined by $\M_0^{fib}(X)$ and $\Xe$. More precisely, if $I$, $\Lambda$ and $A$ are such that 
$\M_0^{fib}(X)$ is isomorphic to the matrix semigroup $M(I,S_{\pe^{-1}(0)},\Lambda;A)$ then  $\M(X)$ is isomorphic to the matrix semigroup $M(I,\Gg,\Lambda;\tilde A)$ where 
$$\tilde a_{\lambda,i}(\xi) = \left\{\begin{array}{ll}
a_{\lambda,i} & \mbox{if } \xi \in \aute(0)\\
e & \mbox{otherwise}
\end{array}\right.$$
and $\Gg$ is as in Lemma~\ref{lem-alg-iso}.
\end{prop}

If $\Gamma_0$ is smaller then the situation is more complicated. We refer the reader to \cite{ESBS} for a discussion of this case which is based on the choice of a splitting section for \eqref{eq-SESg}.

\subsection{Separation of idempotents and semi-regularity}
Let $\Cc_{S_{eX}}(\Gg)$ be the subgroup of permutations of $eX$ which commute with the structure group $\Gg=eEe$. Any automorphism $\alpha\in\aut(X)$ preserves $eX$, since it commutes with the elements of $E(X)$, and its restriction to $eX$ is simply $\alpha e$. Furthermore, an automorphism is uniquely determined by its restriction to $eX$, for if $x\in pX$ with $p$ some other idempotent from $L$, then $\alpha(x) = \alpha p e(x) = p\alpha(ex)$. The restriction map 
$$\aut(X)\ni \alpha\mapsto\alpha e\in C_{S_{eX}}(\Gg)$$ is thus injective. 
The system  $(X,\sigma,T)$ is called {\em semi-regular}  if this restriction map is an isomorphism, namely $\aut e =\{\alpha e:\alpha\in\aut\}= C_{S_{eX}}(\Gg)$ \cite{Auslander-Glasner,ESBS}. 
Note that $\Gg^{fib}$ is trivial if the system is almost automorphic. This implies that $\Gg = \Xe$ and hence $C_{S_{eX}}(\Gg)=\Xe$, showing that almost automorphic systems can only be semi-regular if they are equicontinuous. 
On the other hand, if $cr(X)>1$ then the little structure group cannot be trivial \cite{K25} so that  $C_{S_{eX}}(\Gg)$ looks quite different. 

\begin{lemma}\label{lem-ls}
Let $H,G$ be two commuting subgroups of $S_X$, the permutation group of some set $X$. Suppose that both act transitively and that one is abelian. Then $H=G$ and the groups act fixed point freely.
\end{lemma}
\begin{proof}
Suppose $G$ is abelian so that $H$ and $G$ commute with any $g\in G$. Let $h\in H$ and $x\in X$. As $G$ acts transitively, there is $g\in G$ such that $h(x) = g(x)$. It follows that, for all $f\in G$,
$hf(x)=fh(x) = fg(x) =gf(x)$. Thus $h=g\in G$, again by transitivity of $G$. 
Hence $H\subset G$, hence $H$ is abelian. Repeating the argument with the roles of $H$ and $G$ interchanged we find that $G\subset H$. 

Suppose now that $g(x)=x$. Then $g(h(x))=h(x)$ for all $h\in H$. Hence $g=e$ as $H$ acts transitively. 
\end{proof}


Recall that $Z(G)$ denotes the centre of the group $G$.
\begin{lemma} \label{lem-trans} If $(X,\sigma)$ is a  strongly USF  system with no extra spectrum,
then any point in a  singular fibre separates the idempotents of $L=Ee$. 
If, moreover, $\aut$ is abelian 
and $cr(X)>1$ 
then $\aut e = Z(\Gg)$ and $\Gg$ acts fixed point freely. 
\end{lemma}
\begin{proof}
Let $X_\xi=  e\pe^{-1}( \aute (\xi))$. 
Let $x,y\in X_0$ and $p,q\in J_L$. As $\aut(X)$ acts transitively on $X_0$ there is an automorphism $\varphi$ such that $y=\varphi(x)$. Since automorphisms commute with elements of $L$, we have $p(x)=q(x)$ iff $p(y)=q(y)$. This shows that $p$ and $q$ are equal if they agree on $eX$, as  $p$ and $q$ must act as the identity on regular fibres. But then they also agree on $X$, as $p = pe$ and $q=qe$.

Let $\Gg^{sing}$ be the pre-image of $\aute(0)$ under the map 
$\Gg\stackrel{\tpe}\to \Xe$.   
As the system is strongly USF,  ${X_0}$ is an orbit under $\aut$. Therefore $\Gg^{sing}$ is the subgroup of $\Gg$ which preserves 
$X_\xi$ for any $\xi\in \Xe$. 
Restricting its elements to $X_\xi$, we obtain a subgroup of $S_{X_\xi}$ which we denote by $\Gg^{sing}_\xi$. The restriction of the elements of $\aut$ to $X_\xi$, denoted $\aut_\xi$,  also yields  a subgroup of 
$S_{X_\xi}$. Both groups act transitively, $\Gg^{sing}_\xi$ by minimality, and $\aut_\xi$ by assumption. By assumption, $\aut_\xi$ is abelian and therefore equal to $\Gg^{sing}_\xi$ by Lemma~\ref{lem-ls}. This lemma also implies that $\Gg^{fib}$ acts fixed point freely and therefore all of $\Gg$ acts fixed point freely.

By Theorem~\ref{thm-prod}, $\mu$ maps $\Gf$ to the set of  
all functions $\tilde g$ which satisfy covariance. 
Let $h\in \Gg$ and $g\in \Gf$. Then $h g = g h$ iff 
$\forall \xi\in \Xe \forall x\in e\pe^{-1}(\xi)$, 
\begin{equation}\label{eq-h} 
h \tilde g(\xi)(x) =\tilde g(\xi+{\pe}_*(h)) h(x)
\end{equation} 
(here $\tilde g = \mu(g)$).
This equation means that $\tilde g(\xi+(\pe)_*(h))$ is uniquely determined by $\tilde g(\xi)$ through $h$. But since  $\tilde g(\xi+(\pe)_*(h))$ and $\tilde g(\xi)$ are only constrained through covariance we can chose arbitrary elements from $\Gf_\xi$ and $\Gf_{\xi+(\pe)_*(h)}$ for   $\tilde g(\xi)$ and $\tilde g(\xi+(\pe)_*(h))$ if ${\pe}_*(h)\notin \aute (0)$. Since $\Gf_\xi$ is non-trivial if $cr>1$ this is a contradiction and hence 
${\pe}_*(h)\in \aute (0)$. A necessary condition that $h$ belongs to the centre of $\Gg$ is then that it belongs to $\Gg^{sing}$.  
On the other hand, if ${\pe}_*(h)\notin \aute (0)$ then  \eqref{eq-h} says that a necessary condition for $g$ to belong to the centre of $\Gg$ is that 
$\tilde g(\xi+{\pe}_*(h))=h\tilde g(\xi)h^{-1}$. Composed with the isomorphisms $\Gg^{sing}_\xi=\aut_\xi\cong\aut$ we can see $\tilde g$ as a covariant function from $\Xe$ to $\aut$ and then the last equation holding for any $h$ implies that $\tilde g$ must be constant. Hence $Z(\Gg)=\aut e$.
\end{proof}
Note that under the assumptions of Lemma~\ref{lem-trans},  the system $(X,\sigma)$ is semi-regular.

\begin{cor}\label{cor:fcr}
Let $(X,\sigma,T)$, $(X',\sigma',T)$ be  strongly USF systems with no extra spectrum.
Assume that their coincidence ranks are  strictly larger than $1$, and that they have an abelian automorphism group. If $(X',\sigma')$ is related to  $(X,\sigma)$ by an Ellis morphism then 
$(X',\sigma')$ is a factor of $(X,\sigma)$.
\end{cor}
\begin{proof}
Lemma~\ref{lem-trans} tells us that  any point in a  singular fibre separates the idempotents of $L=Ee$, so 
Assumption (1) of Theorem~\ref{thm-main} is assured.
  
The second part of  Lemma~\ref{lem-trans} implies that $\Phi$ maps $\aut$ to $\aut'$, as a group epimorphism maps centers onto centers. It follows that $\tPhe(\aute)=\aute'$ which implies assumption (2) of Theorem~\ref{thm-main}. 
Finally, Lemma~\ref{lem-trans} says that $\Gg$ acts fixed point freely. Therefore all its point stabilizers are trivial. Hence  assumption (3) of Theorem~\ref{thm-main} is trivially satisfied.
\end{proof}
Almost automorphic systems are not covered by the last corollary, as in this case we cannot argue with the center of the structure group to insure Assumption (2) of Theorem~\ref{thm-main}. Nevertheless we obtain a similar result under a stronger assumption.
\begin{cor}\label{cor:cr1}
Let $(X,\sigma,T)$, $(X',\sigma',T)$ be  strongly USF systems with coincidence rank $1$. 
 Suppose that  $\aute'$ is trivial, i.e. the singular points of $X'_{eq}$ are a single $T$-orbit. 
If $(X',\sigma')$ is related to $(X,\sigma)$ by an Ellis morphism, then $(X',\sigma')$ is a factor of $(X,\sigma)$.
\end{cor}
\begin{proof}
As the coincidence rank equals 1, both systems have no extra spectrum.  As in the proof of Corollary~\ref{cor:fcr} we see that Assumptions (1) and (3) are satisfied. 
 We claim that Assumption (2) follows from the requirement that  
${X'}_{eq}^{sing}$ is a single orbit under the $T$-action. By $T$-equivariance $\tPhe $ then maps the $T$-orbit of $\pe(x)$ to the $T$-orbit of $\pe(x')$. This forces regular points in $\Xe$ to be mapped to regular points. \end{proof}

\section{Examples}\label{sec:Examples}
We provide examples to which our theory can be applied. This is to say that inside the class of these examples we can say that two systems are conjugate if and only if their Ellis semigroups are {related by an Ellis  isomorphism}. As Ellis semigroups are usually more complicated than dynamical systems, the strength of our result does not so much lie in the fact that we can prove that two systems are conjugate or not by comparing their Ellis semigroups. But it is interesting to see that there are non-conjugate dynamical systems which have \emph{algebraically} isomorphic Ellis semigroups. Our theory then tells us that the algebraic isomorphism cannot be continuous. 

Many of our examples are substitutional dynamical systems. For a brief recap of these systems, see Section~\ref{sec:substitutions}.

\subsection{Rudi and Morse}
The Thue-Morse substitution $a\mapsto ab$, $b\mapsto ba$ and the 
Rudin-Shapiro substitution $a\mapsto ac$, $b\mapsto dc$,  $c\mapsto ab$, $d\mapsto db$ define shift dynamical systems $(X_{RS},\sigma)$ and $(X_{TM},\sigma)$ which satisfy all 
conditions of Corollary~\ref{cor:fcr}. For, it can be verified that they have concidence rank $2$. By Lemma~\ref{lem:unique-singular-fibre}, they are a USF system.  Furthermore, each of their automorphism groups is a semidirect product of $\Z/2\Z$ with  $\Z$, so they are strongly USF. That they have no extra spectrum can be seen by Examples~\ref{ex:RS} and ~\ref{ex:TM},  and Corollary~\ref{cor-main-qbij}.

 As we will see in Section ~\ref{ESQB}, their Ellis semigroups are algebraically isomorphic. However, as is well known, these two systems are not conjugate. Indeed, while the continuous component of the spectrum of the Thue-Morse system is purely singular, that of the Rudin-Shapiro system is purely absolute \cite{Baake-Grimm}. Alternatively, one can compute explicitly that there is no conjugacy between them, as in~\cite{CQY}. It therefore follows from Corollary~\ref{cor:fcr} that the 
algebraic isomorphism cannot be continuous. 

\subsection{$(\al, \cut)$-Sturmian systems (2-cut Sturmian systems)}
\label{ex:Sturmian} 
Consider three real numbers $\al$, $\cut$, $\phi$. Taken modulo $\Z$ we interpret $\al$ and $\phi$ as angles in $\mathbb S^1=\R/\Z$ and call 
$\al$ the {\em rotation} angle and $\phi$ the {\em parameter} angle. $\cut$ is supposed to lie in $(0,1)$ and is called the cut value. These systems have been studied in \cite{PS}.

An $(\al,\cut)$-Sturmian sequence with parameter $\phi$ is a two-sided sequence in $\{a,b\}^\Z$ of the form $\omega_\phi^+$ or $\omega_\phi^-$ where 
$$\omega_\phi^+(n) = \left\{\begin{array}{ll}
\!\!	a & \!\! \mbox{if } 0 \leq \{n\al +\phi\} < \cut \\
\!\!	b &\!\! \mbox{otherwise} \end{array}
\right.\!\!\!\!\!\! \mbox{ , and }\:\:\:\:
\omega_\phi^-(x) = \left\{\begin{array}{ll}
\!\!	a & \!\!\mbox{if } 0 < \{n\al +\phi\} \leq \cut \\
\!\!	b &\!\! \mbox{otherwise.} \end{array}
\right.$$
We look here only at the cases in which $\alpha$ is irrational so that the sequences are aperiodic. Note that if $\{n\al+\phi\}\notin \{0,\kappa\}$ then 
$\omega_\phi^+(n) = \omega_\phi^-(n)$. The set of these sequences form a closed shift invariant subspace $X_{\al,\cut}$ of $\{a,b\}^\Z$ (with product topology) on which the shift acts minimally. The associated minimal dynamical system $(X_{\al,\cut},\sigma)$ has the following well-known properties:
\begin{enumerate}
\item Its MEF is $(S^1,+\al)$ with the Euclidean topology and factor map $\pe:X_{\al,\cut}\to S^1$ is given by $\pe(\omega_\phi^\pm) = \phi$. 
A point $\phi$ is singular if $\omega_\phi^+(n) \neq \omega_\phi^-(n)$ for some $n$. The latter happens if $n\al+\phi = 0$ or $n\al+\phi =\kappa$. 
Hence a singular fibre contains two points whereas a regular fibre contains a single point.
\item If $\cut \in \al\Z+\Z$, we say that $\kappa$ is of type $1$, and then there is only one orbit of singular points and the system is strongly USF. Otherwise there are two orbits of singular points.
\item If $\cut \in \frac12(\al\Z+\Z)$, we say that $\kappa$ is of type $2$, then the automorphism group $\aut$ is generated by $\omega_\phi^\pm \mapsto \omega_{\phi+\cut}^\pm$. In particular, 
$\aute(0) = \frac12(\al\Z+\Z)$ mod $\Z$ and the system is strongly USF.  
\item In all other cases we say that $\kappa$ is of type $3$. There are then two orbits of singular points but $\aut = \Z$ and the system is not USF. 
\item If $\omega_\phi^+$ and  $\omega_\phi^-$ are not equal they form a proximal pair. 
Moreover, they differ only at one or two positions, i.e., they are {\em bi-asymptotic}. These are the only proximal pairs.
\end{enumerate}
We now describe the Ellis semigroup. Given that the system is almost automorphic and all proximal pairs are asymptotic, the Ellis semigroup is {\em nearly simple} \cite{Barge-Kellendonk} and must have the algebraic form
$$E(X_{\al,\cut}) \cong \Z\sqcup ( LZ_n\times S^1) $$
where $LZ_n$ is the left zero semigroup of $n$ elements, with $n$ being the number of minimal idempotents, and the product is defined by
$$ \sigma^n (p,z) = (p,z+\al),\quad (p,z)(q,w) = (p,z+w). $$
As there are at most 2 singular fibre orbits and a singular fibre has 2 points, there can be at most $n=4$ idempotents. In particular, the system is tame, i.e., its Ellis semigroup has cardinality $|\mathbb R|$.
 
The calculation of the Ellis semigroup for Sturmian systems was discussed in \cite[Example 14.10]{GlasnerMegrelishvili2006} and generalised in \cite{Aujogue}.  
A sequence $(\sigma^{n_\nu})_\nu$ converges in the pointwise topology in the following three situations:
\begin{enumerate}
\item The sequence of integers $(n_\nu)_\nu$ is bounded and hence converges to an integer $n$. Then $\lim \sigma^{n_\nu} = \sigma^n \in \Z$.
\item The sequence of integers $(n_\nu)_\nu$ is unbounded but the sequence of angles $\al n_\nu\in S^1$ converges to a point $z\in S^1$ from above, that is, almost all  $\al n_\nu\geq z$ (w.r.t.\ the order in a small neighbourhood of $z$ which is induced by $\R$). In this case  $\lim \sigma^{n_\nu} = (p^+,z)$ where $p^+$ is the idempotent given by $p^+(\omega_\phi^\pm) = \omega^+_\phi$.
\item The sequence of integers $(n_\nu)_\nu$ is unbounded but $\al n_\nu\in S^1$ converges to a point $z\in S^1$ from below. 
In this case  $\lim \sigma^{n_\nu} = (p^-,z)$ where $p^-$ is the idempotent given by $p^-(\omega_\phi^\pm) = \omega^-_\phi$.
\end{enumerate}
The topology of $E(X_{\al,\cut})$ can be described by saying which sequences in $E(X)$ converge. A sequence $(p_\nu,z_\nu)\subset LZ_2\times S^1$ converges to a point $(p,z)\in LZ_2\times S^1$ if either $z_\nu$ converges to $z$ from above, in which case we have $\lim (p_\nu,z_\nu) = (p^+,z)$, or $z_\nu$ converges to $z$ from below, in which case we have $\lim (p_\nu,z_\nu) = (p^-,z)$. Furthermore, a sequence $(\sigma^{n_\nu})_\nu$ converges to $\sigma^n$ or 
$(p^+,z)$ or $(p^-,z)$ under the same conditions on $(n_\nu)_\nu$ as above. 

We thus see that $E(X_{\al,\cut})$ depends only on $\al$. 

We see from the above analysis that there are always only 2 idempotents and that points in singular fibres separate idempotents. As the system is almost automorphic, Assumption (3) of Theorem~\ref{thm-main} is  alsosatisfied. However, Assumption (2) holds only if $\cut'$ is of type 1, or $\cut$ and $\cut'$ are of type 2. In these cases we can apply Theorem~\ref{thm-main} to deduce the existence of a factor map from $(X_{\al,\cut},\sigma)$ to $(X_{\al,\cut'},\sigma)$.

\subsection{Coverings of substitutions}\label{sec:coverings}
We thank Lorenzo Sadun for sharing with us preliminary results about the following construction, which works in fact in much greater generality \cite{Sadun}.
Let $\theta$ be an aperiodic primitive substitution on the alphabet $\Aa$, let $G$ be a finite group and 
$q:\Aa^*\to G$ a morphism. We suppose that, for all $a\in \Aa$, 
\begin{equation}\label{eq-cond-q}
   q(\theta(a)) = q(a).    
\end{equation}
We define a new substitution $\tilde \theta$ on $\tilde\Aa:=\Aa\times G$ as follows. If $\theta(a) = a_1\cdots a_n$ then
$$\tilde\theta(a,g) = (a_1,g_1)\cdots (a_n,g_n),\quad g_{k+1} = g_k q(a_k),\quad g_1 = g.$$
Using \eqref{eq-cond-q},
 one verifies that if 
$\theta(ab) = a_1\cdots a_{n+m}$ then  
$\tilde\theta((a,g)(b,gq(a)) = (a_1,g_1)\cdots (a_{n+m},g_{n+m})$ with $ g_1 = g$ and $g_{k+1} = g_k q(a_k)$ for  $1\leq k \leq n+m-1$.  Let $X$ be the shift space defined by $\theta$ and $\tilde X$ be the shift space defined by $\tilde \theta$. This implies that the substitution shift $\tilde X$ is a $|G|$-to-1 cover of $ X$, notably  $\tilde X\cong  X\times G$, with a conjugacy being given by $(a_n,g_n)_n\mapsto ((a_n)_n,g_0)$. Here we have the shift action on the left whereas on the right the action is $((a_n)_n,g_0)\mapsto ((a_{n+1})_n,g_0q(a_0))$. The examples in  \cite{BARGE-BRUIN-Jones-Sadun:2012} are all of this nature.

If $\aut$ is the automorphism group  of $ (X,\sigma)$, then the automorphism group $\tilde\aut$ of $(\tilde X,\tilde\sigma)$ contains $\aut\times G$, where the second factor acts by left multiplication on the second factor of $X\times G$. Note that this means that there are $h$-to-1 covers of $ X$ which are not obtained in this way. Indeed, there are constant height $h>1$ suspensions $\tilde X$ over a constant length substitution system $ X$ where $\tilde\aut=\aut=\mathbb Z$ \cite{CQY}. 

Different choices for $G$ and $q$ lead to different systems, indeed, even the coincidence rank may change. Of course, $q(a) = e$, with $e$ the neutral element of $G$, is always a possibility, but this one is not very interesting, as $\tilde X$ then decomposes into a disjoint union of copies of the original system. 
We are interested in the cases in which the new substitution is primitive, so that the cover is a minimal system. This requires that $q(\Aa)$ generates $G$, but the latter does not guarantee minimality. For example, take the substitution $a\mapsto abb, b \mapsto acb , 
c \mapsto abc$. Take $G= \mathbb Z/3 \mathbb Z$, and take $q(a)=a(b)=q(c)=1$. Then for any element of $\tilde X$, one only ever sees one fixed $(a,i)$ with $i$ depending on $x$.
We call $(G,q)$ non-trivial if $\tilde\theta$ is primitive. 
It may be that one has to go to a higher power of the substitution to obtain a non-trivial $(G,q)$. 
Recall that one of our assumptions is that $X$ is strongly USF. 
The following lemma shows that this property is inherited by  coverings of substitutions. Note that we do not need  $(G,q)$ to be non-trivial. We use $cr(\theta)$ to denote the coincidence rank of $X$.

\begin{lemma} Suppose that $\tilde X$ is a covering of $X$ defined by $(G,q)$.  There a natural number $h$ (the created height) such that 
$ \tilde X_{eq}$ is an $h$-fold covering of $\Xe$ and $ cr(\tilde X) = cr (X)\frac{|G|}{h} $.
Furthemore, if $(X,\sigma)$ is strongly USF then $(\tilde X,\sigma)$ is also strongly USF.
\end{lemma}
\begin{proof} 
Fix a minimal idempotent $e\in E(X)$ and set $\tilde e = e\times \tt{e}$  where $\tt{e}$ is the neutral element in $G$.   
Let $pr$ denote the projection map, in the first coordinate, sending $\tilde X$ to $X$.
By maximality of the MEF there is a factor map $pr_{eq}$ making 
the  diagram 
\[
\begin{tikzcd}
\tilde X \arrow{r}{pr} \arrow[swap]{d}{\tilde \pi_{eq}} & X \arrow{d}{\pe} \\
\tilde X_{eq}\arrow{r}{pr_{eq}}  & \Xe
\end{tikzcd}
\]
commute.
As $\tilde X_e$ is equicontinuous, and $\tilde X$, $X$ have finite coincidence rank, ${pr_{eq}} $ is $h$-to-$1$ for some $h$  \cite{Sacker-Sell:1974}. Furthermore $pr$ is $|G|$-to-$1$. It follows that the size of the fibre $\tilde\pi_{eq}^{-1}(\tilde \xi)$ is $|G|/h$ times the size of the fibre $\pe^{-1}(\xi)$ where $\xi = pr_{eq}(\tilde \xi)$. This implies the first statement. Since singular fibres are characterised by having strictly more points than the coincidence  rank, this also shows that the pre-image under $pr$ of $eX^{sing}$ is precisely $\tilde e \tilde X^{sing}$. Since $\aut\times G\subset \tilde\aut$, any lift under $pr$ of an orbit of $\aut$ lies in an orbit of $\tilde\aut$.  This implies that $\tilde e \tilde X^{sing}$ is contained in an orbit of $\tilde\aut$. But this orbit cannot be larger, as automorphisms preserve the singular points. 
\end{proof}
\begin{lemma}\label{lem-K}
Suppose that $(G,q)$ is non-trivial and $\tilde\aut = \aut\times G$. Then there is a normal subgroup $K$ of $G$ such that $\tilde\aut^{fib} = \aut^{fib}\times K$ and $|G/K| = h$, the created height.   
\end{lemma}
\begin{proof}
By minimality, $\Ef(\tilde X)$ acts transitively on the MEF fibres. As $\tilde\aut^{fib}$ commutes with $\Ef(\tilde X)$ its restriction to $\tilde e\tpe^{-1}(\tilde \xi)$ must belong to the centraliser of $\tilde e\Ef_{\tilde \xi}(\tilde X)\tilde e$ in the group of permutations of this fibre. 
We claim that it is also a subgroup of $\aut^{fib}\times G$. This then implies that there must be  a subgroup $K$ of $G$ such that $\tilde\aut^{fib} = \aut^{fib}\times K$. As $\tilde\aut^{fib}$ is normal, $K$ is normal.  

It remains to prove the claim that $\tilde\aut^{fib}$ is a subgroup of $\aut^{fib}\times G$.  By assumption, $\tilde\aut=  \aut\times G$.
Consider an element $(\alpha,g)$ of 
$ \tilde\aut^{fib}$. An MEF-fibre of $\tilde X$ projects to an 
MEF-fibre of $X$. Hence $\alpha_{eq}=0$, i.e., $\alpha\in\aut^{fib}$.
\end{proof}
To study how the coincidence  rank changes we consider proximal pairs of $X$. This is particularly simple if $X$ is almost automorphic so that proximal pairs are bi-proximal, and furthermore if restrict to the case where any bi-proximal pair is  bi-asymptotic, i.e. we restrict to almost distal systems.
In a bi-asymptotic pair, the two sequences disagree on two words of the same length, say $w^+$ and $w^-$. We can lift these pairs to $|G|$ pairs in $\tilde X$. Now a lifted pair is bi-asymptotic if and only if $q(w^+) = q(w^-)$. If this is the case for all bi-asymptotic pairs, then $q(w^+)q(w^-)^{-1}=e$ and $\tilde cr = cr$. Otherwise we look at the group generated by $\omega:=q(w^+)q(w^-)^{-1}$. If its order is $k$ then $\tilde X$ has a $2k$-cycle in the sense of Barge-Diamond \cite{Barge-Diamond:2007}. It follows that $cr(\tilde\theta) \geq k$. 
We provide some examples.
\begin{example}
The second power of the period doubling substitution $a\mapsto abaa$, $b\mapsto abab$ has column number $1$ and a single proximal pair. The pair differs on one letter,  $w^+=a$, $w^-=b$, so $\omega = q(b)^{-1} q(a)$. Now \eqref{eq-cond-q} is equivalent to $q(aba) = e$ thus $q(b) = q(a)^{-2}$ and $\omega=q(a)^3$. If we require $G$ to be generated by $q(a)$ and $q(b)$ then it must be a cyclic group, say of order $n$. It follows that the coincidence rank $\tilde c$ is the number of orbits of $(+3)$ in $\Z/n\Z$.
\begin{itemize}
 \item If $n=3m$ then created height is $h=3$ and $\tilde c = m$.
 \item For all other n we have $h=1$ and $\tilde c = n$. Specifically  if $n=2$ the resulting $\tilde \theta$ is the Thue-Morse substitution. 
\end{itemize}
\end{example}
\begin{example}
The third power of the Fibonacci substitution is $a\mapsto abaab$, $b\mapsto aba$. Conjugating with $ab$ this substitution defines the same dynamical system as $a\mapsto aabab$, $b\mapsto aab$; this will make it easier for us to constrain $q$.
For any $G$, \eqref{eq-cond-q} implies that $q(a)^2=q(ab)^2 = 1$ (there is no non-trivial solution for the first two powers). 
Requiring that $q(\Aa)$ generates $G$ we have the following options, where $\tau$ denotes a transposition. They  all yield a primitive substitution $\tilde\theta$
\begin{itemize}
\item[1i] $G=S_2$, $q(a) = \tau$, $q(b)=e$,
\item[1ii] $G=S_2$, $q(b) = \tau$, $q(a)=e$,
\item[1iii] $G=S_2$, $q(a)=q(b) = \tau$,
\item[2] $G=S_2\times S_2$,  $q(a) = 1\times \tau$, $q(b)=\tau\times 1$,
\item[3] $G=D_6\cong S_3$, $q(a) = \tau_{12}$, $q(ab)=\tau_{23}$, i.e., $q(b)=\tau_{123}$.
\end{itemize}
The original substitution $\theta$ has a single proximal pair. This pair  is bi-asymptotic. Since here $w^+=ab$ and $w^-=ba$ we conclude immediately that $\omega=e$ if $G$ is abelian. In that case $h=|G|$.  In case 3, note that $\aut = \mathbb Z$ and it can be verified that $\tilde\aut = \aut\times G$. Since $\omega$ generates $A_3\subset S_3$, this reduces the height to  $h=2$.

The Fibonacci substitution shift  is an  example of a $(\al,\cut)$-Sturmian systems, notably with $\al=\cut= {\frac{1-\sqrt{5}}2}$. Hence its Ellis semigroup is 
$\Z\sqcup LZ_2\times S^1$ with $\Z$-action on $S^1$ given by rotation by $\al$.
Cases 1i,1ii,1iii can be compared using Theorem~\ref{thm-main}: they are isomorphic if and only if their Ellis semigroups are isomorphic. It is easy to see that the Ellis semigroups for these three cases are all algebraically isomorphic to $\Z\sqcup LZ_2\times \tilde X_e$ and $ \tilde X_e$ is a $2$-fold covering of $S^1$. Thus $ \tilde X_e
\cong S^1$ but the $\Z$-action on it is given by rotation by $\frac{\al}2$.
\end{example}

\section{The Ellis semigroup of quasi-bijective substitutions}\label{ESQB}
In this section we generalise the work of \cite{ESBS} on the computation of the algebraic structure of the Ellis semigroup to the larger class of dynamical systems obtained from a quasi-bijective substitution. In particular, we compute the Ellis semigroup of the Rudin-Shapiro shift and find that it is algebraically isomorphic to the Ellis semigroup of the Thue-Morse shift. As the computations may well be of independent interest, we keep this section self-contained. 

\subsection{More semigroup preliminaries}\label{preliminaries}
%
In this section $S$ denotes a semigroup without a zero element which admits a minimal idempotent $e$. Then $\Hh_e := e S e$ is a group whose identity element is $e$. The group inverse of an element $s\in \Hh_e$ is the unique normal inverse of $s$ in $S$, 
we denote it by $s^{-1}$. A little care has to be taken with this notation. 
While $(s_1s_2)^{-1} = s_2^{-1} s_1^{-1}$ if $s_1$ and $s_2$ belong to the same group, this may no longer be true if $s_i \in \Hh_{e_i}$ where $e_1,e_2$ are two distinct minimal idempotents. If $S$ is a sub-semigroup of $\Ff(X)$, the semigroup of functions on a space $X$, and $s\in\Hh_e$, then the restriction of $s^{-1}$ to
  $\im e$  is the inverse of the restriction of $s$ to $\im e$ as a function.
Let $S$ be a semigroup 
and let $e_+$, $e_-$ be minimal idempotents of $S$; they may coincide. 
We denote $\Hh_{+}:=e_+S e_+$, respectively $\Hh_{-}=e_-Se_-$,  the group whose identity is $e_+$, respectively $e_-$; both groups are isomorphic. Of relevance for the analysis below is the sub-semigroup of $S$ which is generated $\Hh_{+}$ and $\Hh_{-}$. There are four possibilities: $\Hh_{+}=\Hh_{-}$, $\Hh_+$ and $\Hh_-$ are in the same left ideal, $\Hh_+$ and $\Hh_-$ are in the same right ideal, or none of these first three cases. Let $e_{+-}$ be the unique idempotent in the group $\Hh_{{+-}}=R_{e_+}\cap L_{e_-}$, the intersection of the right ideal of $e_+$ with the left ideal of $e_-$.  It is the unique idempotent such that $e_+e_-\in \Hh_{{+-}}$. We have the following commuting diagram, where $\rho_x$ denotes right and $\lambda_x$ left multiplication: 
\[
\begin{array}{rcl}
\Hh_+  & \stackrel{\rho_{e_{+-}}} \longrightarrow & \Hh_{+-}\\
\lambda_{e_{-+}} \downarrow & & \downarrow \lambda_{e_{-}} \\
\Hh_{-+}  & \stackrel{\rho_{e_{-}}} \longrightarrow & \Hh_{-}
\end{array}
\]
and relations
\begin{eqnarray*}
e_+ e_{+-} = &e_{+-}& = e_{+-}e_- , \\ e_{+-} e_{+} = &e_+& = e_{+}e_{-+} ,  \\ e_{-}e_{-+}  = &e_{-+} &= e_{-+}e_{+} , \\
e_{-} e_{+-} = &e_{-}& = e_{-+}e_{-};
\end{eqnarray*}
we will use these repeatedly.

Let $S^{(2)}\subset \Hh_{-+}\times \Hh_{+}$ and $\phi\in \Hh_{+-}$. Let $\Lambda := \{+,-\}$ be a set of two elements. Define
\begin{equation}\label{eq:r-set}I = \{LR^{-1}:(L,R) \in S^{(2)}\}\subset \Hh_{-+},\end{equation}
and  define the $\Lambda\times I$ matrix $A_\phi=(a_{\lambda \ii})_{\lambda \ii}$ with entries in $\Hh_+$ by 
\begin{equation}\label{eq:matrix} a_{+\,\ii} = e_+ \qquad a_{-\,\ii} = \phi \ii \end{equation}
These data define the matrix semigroup $M[I,\Hh_{+},\Lambda;A_\phi]$. 
Different choices for $\phi$ lead to isomorphic semigroups\footnote{If $\phi=e_+\eta_o^{-1} e_-$ for some $\eta_0\in I$ then the matrix semigroup is normalised.}.


\subsection{Substitutions}\label{sec:substitutions}

We first briefly summarise the notation and results concerning substitutions that we will need; for an extensive background see \cite{Baake-Grimm} or \cite{Pytheas-Fogg}.

A  {\em substitution} is a map from a finite set $\mathcal A$, the alphabet, to the set of 
nonempty finite words (finite sequences) on $\mathcal A$. We extend $\theta$ to a map on finite words by concatenation: 
\begin{equation}\label{eq-concat}\theta(a_1\cdots a_k) = \theta(a_1)\cdots \theta(a_k),
\end{equation}
and to bi-infinite sequences $\cdots u_{-2} u_{-1} u_0 u_1 \cdots$  as  
\[\theta   (\cdots u_{-2} u_{-1} u_0 u_1 \cdots ) :=   \cdots \theta (u_{-2}) \theta (  u_{-1}    ) \theta  (u_{-1})\cdot   \theta (     u_{0})  \theta (    u_{1}) \cdots \, .\]
Here the $\cdot$ indicates the position between the negative indices and the nonnegative indices. 

We say that $\theta$ is {\em primitive} if there is some 
$k\in \N$ such that for any $a,a'\in \mathcal A$,
the word $\theta^k(a)$ contains at least one  occurrence of $a'$. 
We say that a finite word is {\em allowed} for $\theta$ if it appears somewhere in $\theta^k(a)$ for some $a\in \Aa$ and some  $k\in\N$.

The {\em substitution shift} $( X,  \sigma)$ is the dynamical system where the space $X$  consists of all bi-infinite sequences 
all of whose subwords are allowed for $\theta$. If $\theta$ is primitive, then the same shift space $X$ is generated by any power $\theta^n$,  $n\in \N$.
We equip $X$ with the subspace topology of the product topology on $\Aa^\Z$, making the left shift map $\sigma$  a continuous  $\Z$-action. Primitivity  of $\theta$ implies that $(X,\sigma)$ is minimal.

We say that a primitive substitution is {\em aperiodic} if $X$ does not contain any $\sigma$-periodic sequences. This is the case if and only if $X$ is an infinite space.

\subsection{Constant length and quasi-bijective substitutions }\label{def:quasi-bijective}

The substitution $\theta$ has
{\em (constant) length~$\ell$}  if for each $a\in \mathcal A$, 
$\theta (a)$ is a word of length $\ell$. In this case one can describe the substitution with  $\ell$ maps $\theta_i:\mathcal A \rightarrow \mathcal A$, $0\leq i \leq \ell-1$, such that
\begin{equation}\label{eq-as-perm}\nonumber
\theta(a) = \theta_0(a)\cdots \theta_{\ell-1}(a)
\end{equation}
for all $a\in\Aa$.

 Let $\Z_\ell$ denote the  $\ell$-adic integers, i.e., the inverse limit of cyclic groups $ \varprojlim \Z/\ell^n\Z $.  Let  $\Z_{\bar{\ell},h} :=  \varprojlim \Z/\ell^nh\Z $  and let $1 := (\cdots, 0,0,1)$; addition in $\Z_{\bar{\ell},h}$ is performed with carry.   
If $\theta$ is primitive and aperiodic, then Dekking's theorem \cite{dekking} tells us that   $(\Z_\ell, +1)$ is an equicontinuous factor of $(X, \sigma)$. Furthermore, 
there is an $h$, with $0<h<\ell$, with $h$ coprime to $\ell$, such that 
 $(\mathbb{Z}_{\bar{\ell},h}, +1)$, is  the maximal equicontinuous factor of  $(X, \sigma)$. The integer $h$ is called the {\em height} of $\theta$, and   we say that  $\theta$ has {\em trivial height} if $h=1$.

   We fix the factor map $\pe:X\rightarrow \Z_\ell$ from a primitive aperiodic length-$\ell$ substitution shift $(X,\sigma)$ to $(\Z_\ell,+1)$ with which we work in this article. We will specify it by requiring $\pe(u)=0$ if and only if $u$ is a $\theta$-fixed point. We refer the reader to \cite{dekking} for details.

The {\em column rank} of a length-$\ell$ substitution $\theta$  is defined as the minimal number of distinct letters in the image of a column map of $\theta^{n}$, for some $n$. In other words,
\begin{equation} \label{eq:dolumnnumber}c=c (\theta) := \inf_{j,n} \left\{\lvert {\theta^n}_j(\mathcal A)\rvert  : 0\le j <  \ell^n\right\}.\end{equation} 
For details, see \cite{dekking}. We remark that our notion of column rank agrees with Dekking's original definition of column number when $h=1$.  

A substitution $\theta$ is {\em quasi-bijective} if it has constant length and if for some power,  each of the maps $\theta_i$ has rank $c$, where $c$ is the column rank. When $c=|\mathcal A|$, a quasi-bijection substitution is known as a {\em bijective substitution}. As we are only interested in infinite shifts, we only consider  the case $c\geq 2$.
 
 We say that {the quasi-bijective} $\theta$ is {\em simplified} if
 \begin{enumerate}
 \item{each map $\theta_i$ has rank $c$, and}
 \item each $\theta$-periodic point is a fixed point of $\theta$, so that in particular $e_+:=\theta_0 $ and $e_-:= \theta_{\ell-1}$ are idempotents.
  \end{enumerate}
Given any quasi-bijective substitution $\theta$, property  (2) will be satisfied by a large enough power $\theta^n$ of $\theta$. Indeed,  if $L$ is the lowest common multiple of the least periods of the periodic points, then each periodic point is a fixed point under $\theta^L$. The fixed points are of the form $\theta^{\infty}(a)\cdot\theta^{\infty}(b)$, where $ab$ belongs to 
\[ \mathcal A^{(2)} = \{ ab: a\in \im e_- ,\,   b\in \im e_+, \,  \mbox{ and } ab \mbox{ is an allowed word for $\theta$}\}.\]
A fixed point is uniquely determined by such an $ab$;  
we will use the notation $a\cdot b$ to denote it.  

The following lemma can be verified as in the bijective case \cite{ESBS}.

\begin{lemma}\label{lem:unique-singular-fibre}
Let $\theta$ be a primitive, aperiodic, simplified  quasi-bijective substitution of length $\ell$ with maximal equicontinuous factor map
$\pe:X_\theta\to\Z_{\bar{\ell},h}$, where $\pe^{-1}(0)$ consists of the fixed points of $\theta$. Then the orbit of $0\in \Z_{\bar{\ell},h}$ contains all singular points, 
so that  $(X_\theta,\sigma)$ is USF. 
If two points $x,x'\in X_\theta$ are forward (or backward) proximal then they are forward (or backward) asymptotic. 
 \end{lemma}
The last property of the Lemma~\ref{lem:unique-singular-fibre} means that the dynamical systems associated to quasi-bijective substitution shifts are almost distal.
It follows that the Ellis semigroup $E(X_\theta)$ is the disjoint union of its kernel $\M(X_\theta)$ with $\Z$ \cite{Barge-Kellendonk}.

\subsection{Quasi-bijective substitutions and their Ellis semigroup}\label{bijective}

For USF systems with no extra spectrum, the kernel $\M(X_\theta)$ is determined by the restriction of $\Mm^{fib}(X_\theta):=\Mm(X_\theta)\cap \Ef(X_\theta)$ to the fibre $\pe^{-1}(0)$. We denote this restriction by $\Mm^{fib}_0(X_\theta)$. Indeed, this restriction is a matrix semigroup $M(I,\Gg_0^{fib},\{\pm\};A)$ which we compute below.  $\M(X_\theta)$  is then  isomorphic to the matrix semigroup $M(I,\Gg,\{\pm\};\tilde A)$ where $\Gg$ fits into the exact sequence
\begin{equation*}
\Gg^{fib} \hookrightarrow  \Gg \stackrel{\widetilde\pe} \twoheadrightarrow 
\Z_{\bar{\ell},h}
\end{equation*}
involving the structure group $\Gg^{fib}$ of $\Mm^{fib}(X_\theta)$ and 
$\tilde A_{\lambda,i} \in\Gg^{fib}$ the covariant function on $\Z_{\bar{\ell},h}$ which is equal to $\tilde A_{\lambda,i} \in\Gg^{fib}$
$\Xe=\Z_{\bar{\ell},h}$ is the maximal equicontinuous factor of $X$ \cite{ESBS}. 
 We showed in Theorem~\ref{thm-prod} that if $(X,\sigma)$ has no extra spectrum, then  
$$\Gg^{fib} \cong \{\tilde g:Z\to \Gg_0^{fib}:\tilde g(\eta+1) = \sigma \tilde g(\eta) \sigma^{-1}\},$$
where the right hand side is a group under pointwise multiplication and the isomorphism is given by $g\mapsto \tilde g$ where $\tilde g(\eta)(x) = g(x)$ for all $x\in \pe^{-1}(\eta)$. 

%

\subsection{Computation of $\M^{fib}_0(X_\theta)$}\label{structural-section} 
Recall from Proposition~\ref{prop-alg-kernel} that the algebraic structure of the kernel $\M(X_\theta)$ of the Ellis semigroup is uniquely determined by  $\M_0^{fib}(X_\theta)$ and 
$\Z_{\bar{\ell},h}$ provided that the topological normal closure of $\Gamma_0$,  the restriction of the little structure group $\Gamma$ to $e\pe^{-1}(0)$, contains all elements of the permutation group. As quasi-bijective shifts are almost distal, this then determines the algebraic structure of $E(X_\theta)$. We now compute $\M^{fib}_0(X_\theta)$.

Recall that since $\theta$ has length $\ell$ there are maps  $\theta_i:\mathcal A \rightarrow \mathcal A$ such that
$
\theta(a) = \theta_0(a)\cdots \theta_{\ell-1}(a)
$
for all $a\in\Aa$.
$\theta$ is thus uniquely determined by what we call its {\em expansion}, namely its representation as a concatenation of $\ell$ maps, which we write as
$$\theta = \theta_0|\theta_1|\cdots |\theta_{\ell-1}.$$
We use this bar notation to separate the columns, or  to distinguish the concatenation $ \theta_0|\theta_1$  from the composition $ \theta_0\theta_1$.
It follows from (\ref{eq-concat}) that 
the composition of two substitutions $\theta$, $\theta'$ of length $\ell$ and $\ell'$ over the same alphabet, which we simply denote by $\theta \theta'$, then has  an expansion into
$\ell \ell'$ maps
\begin{equation}\nonumber
\theta \theta' = \theta_0\theta'_0 |\cdots |\theta_{\ell-1}\theta'_0|\theta_0\theta'_1| \cdots |
\theta_{\ell-1}\theta'_{\ell'-1}; 
\end{equation}
in particular, the expansion of $\theta^2$ is given by 
\begin{equation}\label{expansion}
(\theta^2)_0|\cdots |(\theta^2)_{\ell^2-1} = \theta_0\theta_0 | \cdots |\theta_{\ell - 1}\theta_{0}|  \theta_0\theta_1 |  \cdots |\theta_{\ell-1}\theta_{\ell-1} 
\end{equation}
and iteratively we find, for any given $n$ the $\ell^n$ maps $(\theta^n)_i$ corresponding to the expansion of $\theta^n$.

The semigroup $S_\theta$ associated to a constant length substitution $\theta$ is defined to be the semigroup of maps on the alphabet which is generated by the column maps $(\theta^n)_i$, $n\in\N$, $i=0,\cdots,\ell^n-1$ of $\theta^n$ for all $n\geq 1$. We call two elements $L,R\in S_\theta$ \emph{consecutive} if $L=(\theta^n)_{i-1}$ and $R=(\theta^n)_i$ for some $n>0$ and $i>0$.


Given the semigroup $S_\theta$ of a quasi-bijective substitution in simplified form, let $e_+=\theta_0$, $e_-=\theta_{\ell-1}$, and let $S^{(2)}_\theta$ be the set of consecutive elements in $\Hh_{-+}\times \Hh_{+}$. 
We call $G_\theta:= \Hh_{+}$ the structure group of the substitution. Let $I_\theta$ be defined as in~\eqref{eq:r-set} with $S^{(2)}=S^{(2)}_\theta$ and $A_\phi$ be defined as in~\eqref{eq:matrix}.
\begin{theorem}\label{thm-RMG}
Let $\theta$ be a primitive aperiodic quasi-bijective substitution in simplified form. The semigroup $\Sfib$ is isomorphic to $M(I_\theta,G_\theta,\{\pm\};A_\phi)$. An isomorphism $\Phi:M(I_\theta,G_\theta,\{\pm\};A_\phi)\to \Sfib$ is given by $\Phi(\ii ,g,\epsilon):= f_{(\ii ,g,\epsilon)}$  where 
\begin{eqnarray*} f_{(\ii ,g,+)}( a\cdot b) &=&\ii g (b)\cdot g(b)\\
f_{(\ii ,g,-)}(a\cdot b) &=& \ii g\phi(a)\cdot  g\phi (a).
\end{eqnarray*} 
\end{theorem}
We will prove Theorem~\ref{thm-RMG} in several steps.
Combined with Proposition~\ref{prop-alg-kernel} this theorem describes the 
algebraic structure of the Ellis semigroup $E(X_\theta)$ provided that the topological normal closure of $\Gamma_0$ contains all elements of the permutation group. 
The condition in Theorem~\ref{thm-RMG} then implies that $G_\theta$ is the permutation group of $\pe^{-1}(0)$.
\begin{cor}\label{cor-main-qbij}
Let $\theta$ be a primitive aperiodic quasi-bijective substitution in simplified form. Suppose that the normal closure of the little structure group $\Gamma_0$ contains all elements of the permutation group. 
Then $E(X_\theta) = \M(X_\theta) \sqcup \Z$ and $\M(X_\theta)$ is algebraically isomorphic to the matrix semigroup 
$M(I_\theta,\Gg,\{\pm\};\tilde A_\phi)$
where $\Gg$ can be identified with all covariant functions $\tilde g$ on the maximal equicontinuous factor $\Z_{\bar{\ell},h}$ such that $ \tilde g(\eta) \in B(\eta,\eta+\xi)$, $\xi = \tpe(g)$. Moreover, $I_\theta$ is as in~\eqref{eq:r-set} with $S^{(2)}=S^{(2)}_\theta$ and $\tilde A_\phi$ is the function on $\Z_{\bar{\ell},h}$ which takes the value $A_\phi(\eta) = A_\phi$ (given by ~\eqref{eq:matrix})
if $\eta$ lies in $\aute(0)$ whereas the entries of $A_\phi(\eta)$ are equal to $e$ if  $\eta$ is regular. The product of an element of $\M(X_\theta)$ with an element $n\in \Z$ corresponds to left multiplication by $\sigma$ on $\Gg$.
\end{cor}
If $\Gamma_0$ is smaller then similar results as for bijective substitutions can be obtained by generalising the approach of \cite{ESBS} using splitting sections. We will not discuss this here. 

 The first step in proving Theorem~\ref{thm-RMG}  is to show that there are  elements of $\M^{fib}_0(X)$that  act as stated. The next step is to show that $\Phi$ is multiplicative. 
Direct verification shows that $\Phi$ is injective. In the last step we show that $\Phi$ is surjective. We start with the following  lemma.
\begin{lemma} \label{lem-neg}
The set $S^{(2)}_\theta$ is invariant under diagonal right multiplication with elements from $\Hh_{-+}\cup \Hh_{+}$. Furthermore, if 
$\phi \in\Hh_{+-}$ then  
the map $(L,R) \mapsto (L\phi,R\phi)$ is a bijection between $S^{(2)}_\theta$ and the set $S^{(2,-)}_\theta$ of consecutive elements $(L,R)$ which lie in $\Hh_{-}\times\Hh_{+-}$. 
\end{lemma}
\begin{proof} Let $(L,R)$ consecutive elements in $\Hh_{-+}\times \Hh_{+}$. By definition 
there are $n>0$ and $0 < i < \ell^n$ such that $L=\theta^n_{i-1}$ and $R=\theta^n_{i}$. Let $\theta_j$ be a column map. By construction $L\theta_j$ and $R\theta_j$ are consecutive. Iteratively we find that 
 $L\phi$ and $R\phi$ are consecutive for any $\phi\in S_\theta$. Now if $\phi\in \Hh_{-+}\cup \Hh_{+}$ then $L\phi \in \Hh_{-+}$  and $R\phi \in \Hh_{+}$. On the other hand, if $\phi\in \Hh_{+-}$ then  $L\phi\in \Hh_{-}$ and $R\phi\in \Hh_{+-}$. We thus see that then $(L,R) \mapsto (L\phi,R\phi)$ maps $S^{(2)}_\theta$ to $S^{(2,-)}_\theta$. Diagonal right multiplication with $e_-\phi^{-1} e_+$ yields an inverse to this map, indeed, $\phi e_-\phi^{-1} e_+ = e_{+-}e_+ = e_+$ and $e_-\phi^{-1} e_+ \phi= e_- e_{+-} = e_-$.
\end{proof}
\begin{prop} \label{prop-hin}
Let $(\ii ,g)\in I_\theta \times G_\theta$.
Then $\Sfib$ contains the functions 
$f_{(\ii ,g,+ )}$ and $f_{(\ii ,g,-)}$.
\end{prop}

\begin{proof}
Let $(\ii ,g)\in I_\theta \times G_\theta$.
We claim that $(\ii g,g)\in S^{(2)}_\theta$. Indeed, as $\ii\in I_\theta$ there are 
$(L',R')\in S^{(2)}_\theta$ such that $\ii = L'{R'}^{-1}$. Thus
  $\ii g = L'{R'}^{-1}g = L'(g^{-1} R')^{-1}$. Clearly, $g = R' (g^{-1} R')^{-1}$ and our claim follows from Lemma~\ref{lem-neg}.
There are thus $n>0$ and $1\leq i\leq \ell^n-1$ such that
$ \ii g =(\theta^n)_{i-1}\in \Hh_{-+}$ and $g =(\theta^n)_i\in \Hh_+$.
 
Let $a\cdot b$ be a fixed point. As $\theta$ is simplified we have $e_{-}(a) = a$ and $e_{+}(b) = b$.
The two-letter word $\sigma^{i}(a\cdot b)_{[-1,0]}$ is $(\theta^n)_{i-1}(b)(\theta^n)_i(b)$. Furthermore 
the expansion of 
$\theta^{k+1}\theta^n$ contains $\theta^k\theta_{\ell-1} (\theta^n)_{i-1}| \theta^k\theta_0 (\theta^n)_{i}$ at positions  
 $[\ell^{k+1}i- \ell^k,\ell^{k+1}i+\ell^k-1]$.
Hence 
$$\sigma^{i\ell^{k+1}}(a\cdot b)_{[-\ell^k,\ell^k-1]} =\theta^k (\theta_{\ell-1} \ii g(b)) \,\theta^k (\theta_0  g(b)) = \theta^k (\ii g(b))\, \theta^k (g(b))$$
(the r.h.s.\ is the concatenation of two words of length $\ell^k$)  
the second equality following from $\theta_{\ell-1} \ii=e_-\ii = \ii$ and $ \theta_0 g=e_+ g = g$. 
It follows that    
$$\sigma^{i \ell^{k+1}}(a\cdot b)   \stackrel{k\to +\infty}\longrightarrow 
\ii g(b)\cdot g(b)$$
in the topology of $X$.
By compactness there exists a map $f_{(\ii ,g,+)}\in E(X)$ which agrees with the map 
$a\cdot b\mapsto \ii g(b)\cdot g(b)$ on the fibre $\pe^{-1}(0)$. An element of $E(X)$ either preserves all $\pe$-fibres or none. Hence 
$f_{(\ii ,g,+)}(a\cdot b)\in \Ef(X)$.

To construct elements in $E(X)$ which act like $f_{(\ii ,g,-)}$, we observe that, by Lemma~\ref{lem-neg}, this is equivalent to finding elements which acts like  
$a\cdot b\mapsto L(a)\cdot R(a)$ on $\pe^{-1}(0)$ where 
$(L,R)\in S^{(2,-)}_\theta$. This can be done as above, the only difference is that we take 
$i' = i-\ell^{n}$ with $n$ and $i$ as above. Then the two-letter word $\sigma^{i'}(a\cdot b)_{[-1,0]}$ is $L(a)\,R(a)$ and we find 
$\sigma^{i' \ell^k}(a\cdot b)   \stackrel{k\to +\infty}\longrightarrow 
L(a)\cdot R(a)$.  
By compactness, we obtain a function $f_{(\ii ,g,-)}\in \Ef(X)$ which restricts on the fixed point fibre to $a\cdot b\mapsto L(a)\cdot R(a) = \ii g \phi(a) \cdot g\phi(a)$.
\end{proof}

Recall the map $\Phi$ defined in Theorem~\ref{thm-RMG}.
\begin{prop}
The map $\Phi$ is a semigroup morphism.
\end{prop}

\begin{proof}  Let $(\ii,g)\in I_\theta\times G_\theta$.
We have 
\begin{eqnarray*}
(\ii,g,+)(\ii',g',\pm) &=& (\ii,ge_+ g',\pm) = (\ii,g g',\pm)\\
(\ii,g,-)(\ii',g',\pm) &=&  (\ii,g \phi\ii' g',\pm)
\end{eqnarray*}
The image under $\Phi$ of the elements of the r.h.s.\ is given by 
\begin{eqnarray*}
f_{(\ii ,gg',+)}(a\cdot b) &=& \ii gg' (b)\cdot gg'(b)\\
f_{(\ii ,gg',-)}(a\cdot b) &=& \ii gg' \phi(a)\cdot g g'\phi(a)\\
f_{(\ii ,g\phi\ii' g',+)}(a\cdot b) &=& \ii g\phi\ii' g' (b)\cdot g \phi\ii' g'(b)\\
f_{(\ii ,g\phi\ii' g',-)}(a\cdot b) &=& \ii g\phi\ii' g' \phi(a)\cdot g \phi\ii' g'\phi(a)\end{eqnarray*}
$\Phi$ applied to the factors of the l.h.s.\ yields
\begin{eqnarray*}  f_{(\ii ,g,+)}f_{(\ii' ,g',+)}(a\cdot b)   &=& 
f_{(\ii ,g,+)}(\ii' g'(b)\cdot g'(b))= \ii gg'(b)\cdot gg'(b)\\
f_{(\ii ,g,+)}f_{(\ii' ,g',-)}(a\cdot b)   &=& 
f_{(\ii ,g,+)}(\ii' g'\phi(a)\cdot g'\phi(a))\\
& =& \ii gg'\phi(a)\cdot gg'\phi(a)\\
f_{(\ii ,g,-)}f_{(\ii' ,g',+)}(a\cdot b)   &=& 
f_{(\ii ,g,-)}(\ii' g'(b)\cdot g'(b))\\
& =& \ii g \phi \ii' g'(b)\cdot g \phi \ii' g'(b)\\
f_{(\ii ,g,-)}f_{(\ii' ,g',-)}(a\cdot b)   &=& 
f_{(\ii ,g,-)}(\ii' g'\phi(a)\cdot g'\phi(a))\\
& =& \ii g\phi\ii' g' \phi(a)\cdot g \phi\ii' g'\phi(a)\end{eqnarray*}
\end{proof}
\begin{prop}\label{prop-rueck}
For any $\varphi\in \Sfib$ there exists $(\ii,g)\in I_\theta\times G_\theta$ such that $\varphi(a\cdot b) = \ii g(b)\cdot g(b)$ or $\varphi(a\cdot b) = \ii g\phi(a)\cdot g\phi(a)$.
\end{prop}
\begin{proof} Any $\varphi\in \Sfib\subset \Ef_0(X)$ is the restriction of a function $f\in \Ef(X)$ to $\pe^{-1}(0)$. We consider first the case that this function belongs to $E(X,\Z^+)$,
that is, $f$ is a pointwise limit of a generalised sequence $(\sigma^{\nu_i})_i$ with $\nu_i> 0$ (the inequality is strict as $f\neq \Id$). Let $ab\in \mathcal A^{(2)}$ be an allowed two-letter word of $\Aa$ which defines a fixed point for $\theta$.  Note that $\varphi(a\cdot b)_{[-1,0]}$ is an open neighbourhood of $\varphi(a\cdot b)$ in $X$. Thus given $a\cdot b\in\pe^{-1}(0)$ there exists a 
$i_0$ such that $\varphi(a\cdot b)_{[-1,0]} = \sigma^{\nu_i}(a\cdot b)_{[-1,0]}$ for $i\geq i_0$. As $\pe^{-1}(0)$ is finite there exists a 
$\nu>0$ such that $\varphi(a\cdot b)_{[-1,0]} = \sigma^{\nu}(a\cdot b)_{[-1,0]}$  for all $a\cdot b\in\pe^{-1}(0)$. Let 
 $L=(\theta^n)_{\nu-1}$ and $R=(\theta^n)_{\nu}$ where $\ell^n\geq \nu$. Then, for all  $a\cdot b\in\pe^{-1}(0)$ we have $\varphi(a\cdot b)_{[-1,0]} = 
L(b) R(b)$. Since the fixed points are uniquely determined by their two-letter word on $[-1,0]$ we have $\varphi(a\cdot b) = L(b) \cdot R(b)$. 
As $\nu>1$ we have $L = Le_+$ and $R=R e_+$. As $L(b) \cdot R(b)$ is a fixed point we have $e_-L=L$ and $e_+R = R$. Hence $(RL^{-1},R)\in I_\theta\times G_\theta$.

If $f \in E^-(X)$ we argue similarly using a generalised sequence $(\sigma^{\nu_i})_i$ to approximate $f_-$ but with $\nu_i<0$.  
This leads to  $\varphi(a\cdot b)_{[-1,0]} = L(a)\cdot R(a)_{[-1,0]}$ with $(L,R)\in S^{(2,-)}_\theta$ which we can write as $(L'\phi,R'\phi)$ by Lemma~\ref{lem-neg}. We therefore get the second equation for $\phi$.
\end{proof}


\begin{example}[Rudin-Shapiro substitution] \label{ex:RS} A simplified form of
the Rudin-Shapiro substitution $\theta$ is given by  
\[
\begin{array}{c} a \mapsto\\ b\mapsto\\ c\mapsto\\ d\mapsto \end{array} 
\begin{array}{c} a\,\\ d\,\\ a\,\\ d\, \end{array}
\!\!\!\!\!\!{\begin{array}{c} c\,\\ b\,\\ c\,\\ b\, \end{array}}
\!\!\!\!\!\!{\begin{array}{c} a\,\\ a\,\\ d\,\\ d\, \end{array}}
\!\!\!\!\!\!{\begin{array}{c} b\,\\ b\,\\ c\,\\ c. \end{array}}
\]
We use the notation $\begin{pmatrix}
           \alpha \\
           \beta \\
           \gamma \\ \delta
          \end{pmatrix} $
to denote the bijection that sends $a$ to $\alpha$, $b$ to $\beta$, $c$ to $\gamma$ and $d$ to $\delta$.
The expansion of $\theta$ is $\theta_0|\theta_1|\theta_2|\theta_3$ with \[\theta_0 = \begin{pmatrix} a\\d\\a\\d \end{pmatrix}=:e_+,  \, 
\theta_1 =  
\begin{pmatrix} 
c\\b\\c\\b \end{pmatrix}, \,\theta_2 =  \begin{pmatrix} a\\a\\d\\d \end{pmatrix} , \,\theta_3 =  \begin{pmatrix} b\\b\\c\\c \end{pmatrix}=:e_-.\]
It follows that $\theta_1 = e_{-+}$ and $\theta_2=e_{+-}$. Furthermore,
\[S_\theta = \left\{  \begin{pmatrix} a\\d\\a\\d \end{pmatrix}, \begin{pmatrix} d\\a\\d\\a \end{pmatrix},   \begin{pmatrix} a\\a\\d\\d \end{pmatrix},  \begin{pmatrix} d\\d\\a\\a \end{pmatrix}, \begin{pmatrix} 
b\\c\\b\\c \end{pmatrix}, \begin{pmatrix} 
c\\b\\c\\b \end{pmatrix},  \begin{pmatrix} b\\b\\c\\c \end{pmatrix},  \begin{pmatrix} c\\c\\b\\b \end{pmatrix}  \right\}.\]
Consecutive elements in $\Hh_{-+}\times \Hh_+$ are $\theta_1\theta_0|\theta_2\theta_0$, $\theta_3\theta_0|\theta_0\theta_1$ and many more, but these suffice to determine that
$S^{(2)}_\theta$ contains $e_{-+}$ and $e_-e_+$ and therefore equals $\Hh_{-+}$. We choose $\phi = e_+ e_-$. With that choice
$a_{-\,e_{-+}} = e_+ e_- e_{-+} = e_+$ and $a_{-\,e_{+}} = e_+ e_- e_{+} =\tau$, the non-trivial element of $\Hh_{-+}\cong S_2$. Hence $\Sfib$ is isomorphic to the semigroup $\mathcal M_2$ which we discussed in Example~\ref{ex:ex}.
\end{example}

\begin{example}[Thue-Morse substitution] \label{ex:TM} A simplified version of the  Thue-Morse substitution is given by
\[
\begin{array}{c} a \mapsto\\ b\mapsto\end{array} 
\begin{array}{c} a\,\\ b \end{array}
\!\!\!\!\!\!{\begin{array}{c} b\,\\ a \end{array}}
\!\!\!\!\!\!{\begin{array}{c} b\,\\ a \end{array}}
\!\!\!\!\!\!{\begin{array}{c} a\,\\ b \end{array}}
\]
This is a bijective substitution. We computed $\Sfib$ in \cite{ESBS}. It turns out that the result is also isomorphic to $\mathcal M_2$. Hence the Rudin-Shapiro shift and the Thue-Morse shift have isomorphic $\M_0^{fib}$ and conjugate equicontinuous factor. It follows from Corollary~\ref{cor-main-qbij} that their Ellis semigroups are algebraically isomorphic. 
\end{example}
\newpage
\section{Notation glossary}\label{sec:notation glossary}

\begin{table}[h]
\begin{tabular}{|l|l|}
\hline
 $\pe:(X,\sigma)\to (\Xe,\delta)$ &  The maximal  equicontinuous factor of $(X,\sigma)$  \\ \hline
 $\aut=\aut(X,\sigma)$ &          The automorphism group of $(X,\sigma)$           \\ \hline
  $\aut^{fib}\subset \aut$ 
& the subgroup of $\aut$ of elements which preserve the $\pe$-fibres. \\ \hline
$\aute:=\{\alpha_{eq}|\alpha\in \aut\}$&  the corresponding automorphisms of $\Xe$  \\ \hline
  $cr=cr(X)$ & the coincidence rank of $(X, \sigma)$  \\ \hline
 $\Ff(X)$  
&  the set of functions from $X$ to itself. \\ \hline
$\ker(E)$ &  the minimal bilateral ideal of $E$  \\ \hline
 $\Gg$ & the Rees structure group of $\ker(E)$ \\ \hline
$\Gamma$ & the little structure group of $E$  \\ \hline
 $\ev_x (f)$
&   evaluation of $f$ at the point $x$ \\ \hline
$\phi_*: E(X) \to E(Y)$  & the factor map induced by  $\phi:(X, \sigma)\to (X',\sigma')$ \\ \hline
 $L^{fib}$ & $\{ f\in L: \widetilde\pe(f)=0\}$  \\ \hline
  $E^{fib} $
& $ \{ f\in E: \widetilde\pe(f)=0\}$      \\ \hline
$\tpe $ &  $\ev_0\circ {\pe}_*:L\to \Xe.$  \\ \hline
 $\Gg^{fib}$ & $\Gg\cap L^{fib}$.  \\ \hline
 $ \Gg^{\xi} $ &   $ \{f\in \Gg: \widetilde\pe (f) = \xi\} $\\ \hline
  $\overline{\text{ncl}_G(H)}$ &  the topological  normal closure of $H$ in $G$  \\ \hline
 $J_L$ & the (minimal) idempotents of $L$ \\ \hline
 $\Xe^{sing}$ &  the set of singular fibres in $\Xe$
  \\ \hline
 $\Xe^{reg}$ &   the set of regular fibres in $\Xe$  \\ \hline
  $X_0$
 & $e\pe^{-1}(\Xe^{sing})$.  \\ \hline
  $\aut_0$ & the restriction of $\aut$ to $X_0$  \\ \hline
 $\res_{A}(f)$ & the restriction of $f$ to the set $A$  \\ \hline
  $\Gf_\xi$ & the restriction of $\Gf$ to $e\pe^{-1}(\xi)$  \\ \hline
  $\Gamma_\xi$ & the restriction of the little structure group $\Gamma$ to $e\pe^{-1}(\xi)$.
  \\ \hline
 $\Gg^{sing}$ &$\{ f\in \Gg: \widetilde{\pe}(f) \in {\mathfrak a}_{eq}(0)\}.$  \\ \hline
  $ \Gg_\xi^{sing} $ & $\res_{e\pe^{-1}  (   {\mathfrak a}_{eq} +\xi      )} (\Gg^{sing}) $  \\ \hline
\end{tabular}
\end{table}

\section*{Acknowledgments}
We thank Lorenzo Sadun for sharing with us preliminary results which led to several examples in Section~\ref{sec:coverings}.

\providecommand{\bysame}{\leavevmode\hbox to3em{\hrulefill}\thinspace}
\providecommand{\MR}{\relax\ifhmode\unskip\space\fi MR }
\providecommand{\MRhref}[2]{%
  \href{http://www.ams.org/mathscinet-getitem?mr=#1}{#2}
}
\providecommand{\href}[2]{#2}

\end{document}